\documentclass[11pt,reqno]{amsart}
\usepackage{amsmath,amsthm,amssymb,enumerate}
\usepackage{hyperref}
\usepackage{mathrsfs}
\usepackage{nicefrac}
\usepackage{mhequ}

\usepackage{color}

\usepackage[a4paper,innermargin=1.4in,outermargin=1.4in,bottom=1.5in,marginparwidth=1.5in,marginparsep=3mm]{geometry}

\numberwithin{equation}{section}

\newtheorem{theorem}{Theorem}[section]
\newtheorem{lemma}[theorem]{Lemma}
\newtheorem{proposition}[theorem]{Proposition}
\newtheorem{corollary}[theorem]{Corollary}

\theoremstyle{definition}
\newtheorem{assumption}{Assumption}[section]
\newtheorem{definition}{Definition}[section]

\theoremstyle{remark}
\newtheorem{remark}{Remark}[section]

\newcommand{\red}[1]{{\color{black}#1}}
\newcommand\bR{\mathbb{R}}
\newcommand\bP{\mathbb{P}}

\newcommand\cE{\mathcal{E}}
\newcommand\cF{\mathcal{F}}

\newcommand{\bT}{\mathbb{T}}
\newcommand{\bN}{\mathbb{N}}
\newcommand{\wto}{\rightharpoonup}
\newcommand{\tq}{\tilde{q}}

\newcommand*{\beq}{\begin{equation}}
\newcommand*{\eeq}{\end{equation}}
\newcommand{\E}{\mathbb{E}}

\newcommand{\R}{\mathbb{R}}
\newcommand{\bit}{\begin{itemize}}
\newcommand{\eit}{\end{itemize}}
\newcommand\D{\partial}
\newcommand{\supp}{\text{supp}\,}

\newcommand{\tu}{{\tilde{u}}}
\newcommand{\tA}{{\tilde{A}}}

\newcommand{\tPsi}{{\tilde{\Psi}}}
\newcommand{\eps}{\varepsilon}
\newcommand{\T}{\mathbb{T}}

\newcommand{\fra}{\mathfrak{a}}
\newcommand{\tfra}{{\tilde{\mathfrak{a}}}}

\renewcommand{\le}{\lesssim}
\newcommand{\tsigma}{{\tilde\sigma}}

\newcommand{\txi}{{\tilde\xi}}
\newcommand{\N}{\mathbb{N}}

\DeclareMathOperator*{\esssup}{ess\,sup}

\begin{document}

\title{Entropy solutions for Stochastic Porous Media Equations}
\author{K. Dareiotis, M. Gerencs\'er, and B. Gess}

\address[K. Dareiotis]{Max Planck Institute for Mathematics in the Sciences, Inselstrasse 22, 04103 Leipzig, Germany}
\email{konstantinos.dareiotis@mis.mpg.de}

\address[M. Gerencs\'er]{Institute of Science and Technology Austria, 
Am Campus 1, 
A – 3400 Klosterneuburg, Austria}
\email{mate.gerencser@ist.ac.at}

\address[B. Gess]{Max Planck Institute for Mathematics in the Sciences, Inselstrasse 22, 04103 Leipzig and Faculty of Mathematics,
University of Bielefeld,
33615 Bielefeld,
Germany}
\email{benjamin.gess@gmail.com}

\begin{abstract}
We provide an entropy formulation for porous medium-type equations with a stochastic, non-linear, spatially inhomogeneous forcing. 
Well - posedness and $L_1$-contraction is obtained in the class of entropy solutions.
Our scope allows for porous medium operators $\Delta (|u|^{m-1}u)$ for all $m\in(1,\infty)$, and H\"older continuous diffusion nonlinearity with exponent $1/2$.
\end{abstract}

\maketitle

\section{Introduction}
We consider degenerate quasilinear stochastic partial differential equations (SPDEs), with the stochastic porous medium equation
\begin{equs}         
d u(t,x) &=  \Delta (|u(t,x)|^{m-1}u(t,x)) \, dt + \sum_{k=1}^\infty\sigma^k(x,u(t,x)) \, d\beta^k(t) &\quad&\text{on }(0,T)\times\T^d, \\
u(0,x) &=\xi(x) &\quad&\text{on }\T^d,   \label{eq:intro_PME}
\end{equs}
for $m\in (1,\infty)$ as a model example. Here $T>0$ denotes the time horizon, $\bT^d$ is the $d$-dimensional torus, $\beta^k$ are independent Wiener processes and $\xi \in L^{m+1}(\bT^d)$. Assuming sufficient regularity on $\sigma^k$ we prove the well-posedness of entropy solutions to \eqref{eq:intro_PME} (cf.  Theorem \ref{thm:main} below). This is the first time that well-posedness for \eqref{eq:intro_PME} can be shown in the full range $m\in (1,\infty)$.

Equation \eqref{eq:intro_PME} is viewed as a special case of a class of SPDE of the type
\begin{equation}
\begin{aligned}          \label{eq:main_equation}
d u(t,x) &=  \Delta A(u(t,x)) \, dt + \sum_{k=1}^\infty\sigma^k(x,u(t,x)) \, d\beta^k(t) &\quad&\text{on }(0,T)\times\T^d, \\
u(0,x) &=\xi(x) &\quad&\text{on }\T^d.
\end{aligned}
\end{equation}
In addition to well-posedness, we show the stability of the solution map of \eqref{eq:main_equation}
with respect to the nonlinearity $A$,  the diffusion terms $\sigma^k$, and the initial condition $\xi$ (cf. Theorem \ref{thm:stability} below).  More precisely, assuming that $A_n$, $\sigma^{k}_n$, and $\xi_n$ satisfy appropriate regularity assumptions and converge to $A$, $\sigma^k$, and $\xi$, respectively,  we show that the corresponding solutions to \eqref{eq:main_equation} converge in $L^1(\Omega\times(0,T)\times\T^d)$.

Stochastic porous media equations \eqref{eq:intro_PME} have attracted considerable interest in recent years (cf.\ e.g.\ \cite{ROK4, ROK2, ROK3, ROK1, GH18, DHV16, GS16-2, Wit} and the references therein) and different approaches, for example based on monotonicity in $H^{-1}$, based on entropy solutions and based on kinetic solutions have been developed.
When applied to the case of Nemytskii type diffusion coefficients as in \eqref{eq:intro_PME}, the established results either lead to strong assumptions on the diffusion coefficients $\sigma^k$ (cf.\ Section \ref{sec:lit} below for more details), or the porous medium equation could not be treated in its full regime\footnote{It should be noted that in other regards, e.g.~boundary data, lower order terms etc., some of the mentioned works address more difficult situations than considered here. We focus only on their applicability to \eqref{eq:intro_PME}.}, e.g.\ \cite{Wit,GH18} were restricted to $m>2$. It is one of the main contributions of this work to dispense of this restriction and to treat \eqref{eq:intro_PME} in the full range $m\in (1,\infty)$  under mild assumptions on the Nemytskii type diffusion coefficients $\sigma$.

The approach to \eqref{eq:intro_PME} developed in this work relies on a notion of entropy solutions similar to \cite{KARLSEN} in the deterministic case. In the class of entropy solutions, we prove an $L_1$-contraction estimate, as well as a generalized $L^1$-stability estimate. The key point of the proof is the derivation of sharp bounds on the errors introduced by variable doubling in the proof of uniqueness/stability. This relies on a careful analysis of the degenerate behavior of the porous medium operator at small values of $u$, by means of an interpolation argument. Due to the aforementioned stability estimates, in the proof of existence of entropy solutions we are able to pass to the limit in approximating equations without using any compactness argument. In particular, unlike many previous works mentioned below, this argument does not rely on probabilistically weak solutions, which may prove useful when dealing with equations with random coefficients.

\subsection{Literature}\label{sec:lit}
Equations of the type \eqref{eq:main_equation} and stochastic porous media equation in particular have attracted a lot of interest in the recent years. Given the vast literature on the subject, we will only mention some of the most relevant works for the present paper and refer to the monographs \cite{ROK3,LR15,DPZ14,ROK1} for a more complete account of the literature.

In \cite{P75,KR79,ROK4, ROK3} a monotone operator approach is employed based on the pivot space $H^{-1}$. When applied to the particular case of Nemytskii type diffusion coefficients, the resulting general abstract conditions could be verified if $\sigma^k$ are affine linear functions of $u$, since otherwise the maps $u \mapsto \sigma^k(u)$ are not known to be Lipschitz continuous in $H^{-1}$, even if $\sigma^k$ is smooth. 
In order to relax this assumption on $\sigma^k$, alternative approaches based on $L^1$-techniques have been introduced. In the deterministic setting, this has been realized via the theory of accretive operators going back to Crandall-Ligget \cite{CL71}  (cf.\ e.g.\ \cite{B10}), 
entropy solutions due to Otto \cite{Otto}, Kru\v{z}kov \cite{KRUZKOV} and kinetic solution by Lions, Perthame, Tadmor \cite{LPT94} and Chen, Perthame \cite{CHEN}. For continuation in this direction  we refer to \cite{Carr}, \cite{KARLSEN}, \cite{BR18} and references therein.  

In the stochastic setting, an entropy formulation was first introduced in \cite{KIM} for stochastic conservation laws.  Works that followed in this direction include \cite{Witt2, FENG, Biswas, KARLSEN2}. A kinetic solution theory was developed in \cite{DV10,DHV16,GH18,GS16-2, FG18,FG17}. For existing work concerning kinetic/entropy solutions to stochastic degenerate parabolic equations we refer  to \cite{HOF1, DHV16}, and to the more recent works \cite{Wit, GH18}. In \cite{Wit}, $A$ is assumed to be globally Lipschitz-continuous. Moreover, when $\sigma$ is Lipschitz continuous in $(x,u)$, a behaviour $A(u) \sim |u|^{m-1}u$ near the origin is allowed only for $m>2$. 
In \cite{GH18}, the condition on the boundedness of $A'$ is dropped, and by using a kinetic formulation,  well-posedness is proved under the condition that $\sigma$ is Lipschitz in $(x,u)$ and $\sqrt{A'(u)}$ is $\gamma$-H\"older continuous with $\gamma>1/2$. In particular, in the case of the porous medium operator, this forces the condition $m>2$. 

A different approach to stochastic porous media equations  based on $L^1$-techniques and also allowing for a first order divergence part has recently been developed in \cite{BR17}.

At last, the interested reader may find various and far reaching treatments of stochastic porous media equations in different settings (various types of boundary conditions, general nonlinearities, nonlocal diffusions, different driving processes, regularity) in the works \cite{ROK1,ROK4,B15,BR15-2,BR15,GT14,G12,BR12,B10,BM09} and the references therein.

\subsection{Notation}
We fix a filtered probability space $(\Omega,(\cF_t)_{t\in[0,T]},\mathbb{P})$, carrying an infinite sequence of
independent Wiener processes $(\beta^k(t))_{k\in\N,\,t\in[0,T]}$.
Introduce the shorthand notations $\Omega_T=\Omega\times[0,T]$, $Q_T=[0,T]\times\T^d$.
Lebesgue and Sobolev spaces are denoted in the usual way by $L_p$ and $W^k_p$, respectively.
When a function space is given on $\Omega$ or $\Omega_T$,
we understand it to be defined with respect to $\cF:=\cF_T$ and the predictable $\sigma$-algebra, respectively. 
In all other cases the usual Borel $\sigma$-algebra will be used.

We fix a non-negative smooth function $\rho:\R\mapsto \R$ which is bounded by $2$, supported in $(0,1)$, and integrates to $1$.
We use the notation $\rho_\theta(r)=\theta^{-1}\rho(\theta^{-1}r)$, which will most often be our mollifier sequence.
When smoothing in time, the property that $\rho$ is supported on positive times will be very convenient.
For spatial regularisation this will be irrelevant, but for the sake of simplicity, we often use $\rho_\theta^{\otimes d}$ for smoothing in space as well.

We will encounter many multiple integrals on a regular basis. To shorten notation, we often write $\int_t$ in place of $\int_0^T\,dt$ (and similarly for $\int_s$),
as well as $\int_x$ in place of $\int_{\T^d}\,dx$ (and similarly for $\int_y$), and $\int_a$ in place of $\int_\R\,da$.
Whenever the integral is taken on a different domain, with respect to a different variable, or is a stochastic integral, we use the usual notation to avoid confusion.
In the proofs we will often use the notation $a \le b$ for $a ,b \in \bR$, which means $a \leq N b$ for some constant $N\geq 0$, the dependence of which on certain parameters is specified in the corresponding 
statement.
Summation with respect to repeated indices (most commonly over $k\in\N$) is often used.

\section{Formulation and main results}
We denote by $\cE(A,\sigma,\xi)$ the Cauchy problem \eqref{eq:main_equation} and set 
\begin{equ}
\fra(r)=\sqrt{A'(r)},\qquad\Psi(r)=\int_0^r\fra(\zeta)\,d\zeta.
\end{equ}
Following \cite{CHEN} we formulate conditions on the nonlinearity $A$ via assumptions on $\Psi$, with some constants $m>1$, $K\geq 1$, which we consider fixed throughout the article.
We also fix $\kappa\in(0,1/2]$, and $\bar\kappa\in((m\wedge2)^{-1},1]$, standing for the regularity exponents for $\sigma$.

\begin{assumption}\label{as:A}
The following hold:
\begin{enumerate}[(a)]
\item \label{as:A first}
The function $A:\R\mapsto\R$ is differentiable, strictly increasing and odd. The function $\fra$ is differentiable away from the origin, and satisfies the bounds
\begin{equ}\label{eq:as fra}
|\fra(0)|\leq K,
\quad |\fra'(r)|\leq K|r|^{\frac{m-3}{2}}\quad\text{if}\ r>0,
\end{equ}
as well as
\begin{equation}\label{eq:as Psi}
K\fra(r)\geq I_{|r|\geq 1},\quad
  K|\Psi(r)-\Psi(\zeta)|\geq\left\{
  \begin{array}{@{}ll@{}}
    |r-\zeta|, & \text{if}\ |r|\vee|\zeta|\geq 1, \\
    |r-\zeta|^{\frac{m+1}{2}}, & \text{if}\ |r|\vee|\zeta|< 1.
  \end{array}\right.
\end{equation} 

\item \label{as:ic}
The initial condition $\xi$ is an $\mathcal{F}_0$-measurable $L_{m+1}(\bT^d)$-valued random variable such that $\E \| \xi \|_{L_{m+1}(\bT^d)}^{m+1}< \infty$.

\end{enumerate}
\end{assumption}

\begin{assumption}   \label{as:noise}
 One of the following holds:
\begin{enumerate}[(a)]
\item  \label{as:sigma}The function $\sigma:\T^d\times\R\mapsto \red{l}_2$
satisfies the bounds, for all $r\in\R$, $\zeta\in[r-1,r+1]$, $x,y\in\T^d$,
\begin{equs}
|\sigma(x,r)|_{l_2}&\leq K(1+|r|),
\\
|\sigma(x,r)-\sigma(y,\zeta)|_{l_2}&\leq K|r-\zeta|^{1/2+\kappa}+K\red{(1+|r|)}|x-y|^{\bar\kappa}.
\end{equs}

\item \label{as:sigma sqrt}
The function $\sigma:\R\mapsto\ell_2$ satisfies the bounds, for all $r\in\R$, $\zeta\in[r-1,r+1]$,
\begin{equ}
|\sigma(r)|_{l_2}\leq K(1+|r|),\quad |\sigma(r)-\sigma(\zeta)|_{l_2}\leq |r-\zeta|^{1/2}.
\end{equ}

\end{enumerate}
\end{assumption}

Let us now introduce the definition of entropy solutions. Set, for $f \in C(\bR)$,
$$
\Psi_f(r) := \int_0^r f(\zeta) \fra(\zeta) \, d\zeta,
$$
so that $\Psi=\Psi_1$.
\begin{definition} \label{def:solution}
An entropy solution of $\cE(A,\sigma,\xi)$ is 
a predictable stochastic process $u : \Omega_T \to L_{m+1}(\bT^d)$ such that
\begin{enumerate}[(i)]

\item  \label{item:in-Lm} $u \in  L_{m+1}(\Omega_T; L_{m+1}(\bT^d))$ 

\item \label{item:chain_ruleW2} For all $f \in C_b(\bR)$ we have $\Psi_f(u) \in L_2(\Omega_T;H^1(\bT^d)$) and 
\begin{equation*}      
\D_i \Psi_f(u)= f(u) \D_i \Psi(u).
\end{equation*}

\item \label{item:entropies}For all convex $\eta\in C^2(\bR)$ with $\eta''$ compactly supported and all  $\phi\geq 0 $ of the form $\phi= \varphi \varrho$ with $\varphi \in C_c([0,T))$, $\varrho \in C^\infty(\bT^d)$,   we have 
almost surely
\begin{align}      
\nonumber          
-\int_0^T \int_{\bT^d} \eta(u) \D_t\phi \, dx dt   &\leq \int_{\bT^d} \eta(\xi) \phi(0)\, dx+\int_0^T \int_{\bT^d} q_\eta(u) \Delta \phi  \, dx dt
\\
\nonumber   
&+\int_0^T \int_{\bT^d}  \left( \frac{1}{2} \phi \eta''(u)|\sigma(u)|_{l_2}^2-\phi \eta''(u) | \nabla \Psi(u)|^2 \right) \,  dx dt  
\\   \label{eq:entropy_inequality}
&+ \int_0^T \int_{\bT^d} \phi \eta'(u)\sigma^k(u) \, dx  d\beta^k(t),
\end{align}
where $q_\eta$ is any function satisfying $q_\eta'(\zeta)=\eta'(\zeta)\fra^2(\zeta)$.
\end{enumerate}
\end{definition}

\begin{remark}
Note that with the particular choices $\eta(r)=\pm r$, \eqref{eq:entropy_inequality} implies that any entropy solution satisfies  equation \eqref{eq:main_equation} in the weak (in both space and time) sense.
\end{remark}

Our main result now states the well-posedness of $\cE(A,\sigma,\xi)$ as follows.

\begin{theorem}\label{thm:main}
Let $A$ and $\xi$ satisfy Assumption \ref{as:A} and let $\sigma$ satisfy  Assumption \ref{as:noise}.
Then there exists a unique entropy solution  to \eqref{eq:main_equation} with initial condition $\xi$.  Moreover, if $\tu$ is the unique entropy to solution \eqref{eq:main_equation} with initial condition $\tilde{\xi}$,  then 
\begin{equ}\label{eq:main contraction}
\esssup_{t\in[0,T]}\E\int_{\bT^d}|u(t,x)-\tu(t,x)|\, dx \leq\E\int_{\bT^d}|\xi(x)-\txi(x)|\, dx .
\end{equ}
\end{theorem}

We also show stability of the solution map with respect to the coefficients, in the following sense.
\begin{theorem}\label{thm:stability}
Let $(A_n)_{n\in\N}$ and $(\xi_n)_{n\in\N}$ satisfy Assumption \ref{as:A} uniformly in $n$, and let $(\sigma_n)_{n\in\N}$ satisfy either Assumption \ref{as:noise} (\ref{as:sigma}) or (\ref{as:sigma sqrt}) uniformly in $n$.
Assume furthermore that, for $n \to \infty$,  $A_n\to A$ uniformly on compact sets, $\xi_n\to\xi$ in $L_{m+1}(\Omega\times\T^d)$, and \red{$\sup_{x,r}(1+|r|)^{-(m+1)}|\sigma_n(x,r)-\sigma(x,r)|^2_{l_2}\to 0$}.
Let $u_n, u$ be  the entropy solutions of $\cE(A_n,\sigma_n,\xi_n)$, $\cE(A,\sigma,\xi)$, respectively.
Then $u_n\to u$ in $L_1(\Omega_T\times\T^d)$, as $n \to \infty$.
\end{theorem}

\section{Preliminaries, strong entropy solutions}

\subsection{Consequences of the setup}
Let us begin with a simple consequence of Definition \ref{def:solution}.

\begin{remark}             \label{rem:W1}
By Definition \ref{def:solution} \eqref{item:in-Lm},  \eqref{item:chain_ruleW2},  it follows that for any $f \in C(\bR)$ satisfying
$$
|f(r)| \leq N (1+ |r|^{\frac{m+1}{2}}),
$$
we have that $\Psi_f(u) \in L_1(\Omega_T;W^{1}_1(\bT^d))$ and 
\begin{equation}         \label{eq:derivative-G-psi}
\D_i\Psi_f(u)= f(u)\D_i\Psi(u).
\end{equation}
Indeed, let $f_n \in C_b(\bR)$ such that $f_n=f$ on $[-n,n]$ and $\sup_n |f_n(r)| \leq N(1+|r|^{\frac{m+1}{2}})$ for all $r \in \bR$. On the basis of Definition \ref{def:solution} \eqref{item:in-Lm},  \eqref{item:chain_ruleW2},  we have 
$$
\E \int_{t,x}| \Psi_{f_n}(u)-\Psi_{f}(u)| \leq N \E \int_{t,x} I_{|u|\geq n} (1+ |u|^{m+1}) \to 0,
$$
and
\begin{equ} \label{eq:derivative-G-conv} 
\E \int_{t,x}|\D_i \Psi_{f_n}(u)-f(u)\D_i\Psi(u)| \leq  N \E \int_{t,x} I_{|u|\geq n}\big(1+|\D_i\Psi(u)|^2 
+  |u|^{m+1}\big)  \to 0,
\end{equ}
as $n \to \infty$. 
Hence the sequence $\Psi_{f_n}(u)$ is Cauchy in $L_1(\Omega_T;W^1_1(\bT^d))$. Since its limit is
$\Psi_{f}(u)$
in  $L_1(\Omega_T;L_1(\bT^d))$, this also holds in $L_1(\Omega_T;W^1_1(\bT^d))$, and
\eqref{eq:derivative-G-conv} implies \eqref{eq:derivative-G-psi}.
\end{remark}
The fact that we can bound the derivative of $\Psi(u)$ for any solution will be particularly useful thanks to the following lemma.
\begin{lemma}\label{lem:frac reg}
Let Assumption \ref{as:A} hold, let $u\in L_1(\Omega\times Q_T)$ and for some $\eps\in (0,1)$, let $\varrho:\R^d\mapsto\R$ be a non-negative function integrating to $1$ and supported on a ball of radius $\eps$.
Then one has the bound
\begin{equ}\label{eq:frac reg}
\E\int_{t,x,y}|u(t,x)-u(t,y)|\varrho(x-y)\leq N \eps^{\frac{2}{m+1}}\big(1+\E\|\nabla\Psi(u)\|_{L_1(Q_T)}\big),
\end{equ}
where $N$ depends on $d, K$ and $T$.
\end{lemma}
\begin{proof}
We may and will assume that the right-hand side is finite. Using $\int_{t,x,y} \varrho(x-y)\leq N$, 
and \eqref{eq:as Psi}, the left-hand side of \eqref{eq:frac reg}
can then be estimated by a constant times
\begin{equs}
&\quad\E \int_{t,x,y}I_{|u(t,x)|\vee |u(t,y)|\geq 1}|u (t,x)- u(t,y)|\varrho(x-y)
\\
&\quad + \Big(\E\int_{t,x,y}I_{|u(t,x)|\vee |u(t,y)|< 1}|u (t,x)- u(t,y)|^{\frac{m+1}{2}}\varrho(x-y)\Big)^{\frac{2}{m+1}}
\\
&\le\E \int_{t,x,y}I_{|u(t,x)|\vee |u(t,y)|\geq 1}|\Psi(u (t,x))- \Psi(u(t,y))|\varrho(x-y)
\\
&\quad + \Big(\E\int_{t,x,y}I_{|u(t,x)|\vee |u(t,y)|< 1}|\Psi(u (t,x))- \Psi(u(t,y))|\varrho(x-y)\Big)^{\frac{2}{m+1}}.\label{eq:00}
\end{equs}
One can now drop the indicator functions, perform the change of variables $z=x-y$, and write
\begin{equs}
\E&\int_{t,x,z}|\Psi(u (t,x))- \Psi(u(t,x-z))|\varrho(z)
\\
&\leq \E\int_{t,x,z}\varrho(z)|z|\int_0^1|\nabla\Psi(u)(x-\lambda z)|\,d\lambda \leq \eps\E\|\nabla\Psi(u)\|_{L_1(Q_T)}.
\end{equs}
Substituting this back to \eqref{eq:00} yields the claim.
\end{proof}
As in \cite{Biswas}, we have that for an entropy solution, the initial condition is attained in the following sense:
\begin{lemma}             \label{lem:initial-condition}
Let $u$ be an entropy solution of \eqref{eq:main_equation}.  Then one has 
$$
\lim_{h \to 0} \frac{1}{h}\E \int_0^h \int_{x} |u(t,x)-\xi(x)|^2 \,dt=0.
$$
\end{lemma}
\begin{proof}
Let $\varrho_{\varepsilon} \in C^\infty(\bT^d)$ be a mollification sequence, for example $\varrho_\eps=\rho_\eps^{\otimes d}$.
We have 
\begin{align}               \label{eq:spliting}
\nonumber
\frac{1}{h}\E \int_0^h \int_{x} & |u(t,x)-\xi(x)|^2  \,  dt 
\leq 2 \E  \int_{x,y} |\xi(y)-\xi(x)|^2 \varrho_\varepsilon(x-y) 
\\
&+ \frac{2}{h}\E \int_0^h \int_{x,y} |u(t,x)-\xi(y)|^2 \varrho_\varepsilon(x-y)\,  dt .
\end{align}
 We first estimate the second term on the right hand side. 
 Take a decreasing, non-negative $\gamma\in C^\infty([0,T])$, such that 
\begin{equ}
\gamma(0)=2,\quad  \gamma\leq 2I_{[0,2h]},\quad\partial_t\gamma\leq -\frac{1}{h}I_{[0,h]}.
\end{equ}
Take furthermore for each $\delta>0$, $\eta_\delta\in C^2(\R)$ defined by
\begin{equ}
\eta_\delta(0)=\eta_\delta^\prime(0)=0,\quad \eta_\delta''(r)=2I_{[0,\delta^{-1})}(|r|)+(-|r|+\delta^{-1}+2)I_{[\delta^{-1},\delta^{-1}+2)}(|r|).
\end{equ}
Let $y\in\T^d$ and $a \in \bR$. Then, using the entropy inequality \eqref{eq:entropy_inequality} with $\phi(t,x)= \gamma(t) \varrho_\varepsilon(x-y)$, $\eta (r)= \eta_\delta(r-a)$, and $q_\eta= \int_a^r \eta'(\zeta) \mathfrak{a}^2(\zeta) \, d\zeta$,   we obtain 
\begin{equs}            
-\int_{t,x} &\eta_\delta(u(t,x)-a) \D_t \gamma(t) \varrho_\varepsilon(x-y)
\leq 2\int_{x}\eta_\delta(\xi(x)-a)\varrho_\varepsilon(x-y) 
\\ 
& +N\int_{t,x} (|u(t,x)|^{m+1}+|a|^{m+1}) |\Delta_x \varrho_\eps(x-y)| \gamma(t) 
\\  
&+\frac{1}{2}\int_{t,x} \eta_\delta^{\prime\prime}(u(t,x)-a)|\sigma(x,u(t,x))|_{l_2}^2 \varrho_\varepsilon(x-y)\gamma(t)
\\   
&+ \int_0^T \int_{x} \eta_\delta^{\prime}(u(t,x)-a)\sigma^k(x,u(t,x))\varrho_\varepsilon(x-y)\gamma(t) \, d\beta^k(t),
\end{equs}
where for the second term on the right hand side we have used \eqref{eq:as fra}. 
Notice that all the terms are continuous in $a \in \bR$.
Upon substituting $a=\xi(y)$ taking expectations, integrating over $y \in \bT^d$, and using the bounds on $\gamma$, one gets
\begin{equs}              
&\frac{1}{h}\int_0^h \E \int_{x,y} \eta_\delta(u(t,x)-\xi(y))  \varrho_\varepsilon(x-y)\,dt
\\
& 
\leq 2\E \int_{x,y}\eta_\delta(\xi(x)-\xi(y))\varrho_\varepsilon(x-y) 
\\ 
& +\frac{N}{\varepsilon^2}\E \int_0^{2h}\int_{x} (|u(t,x)|^{m+1}+|\xi(x)|^{m+1}) \,dt
\\
&+\E \int_0^{2h} \int_{x} \eta_\delta^{\prime\prime}(u(t,x)-\xi(x))|\sigma(x,u(t,x))|_{l_2}^2 \varrho_\varepsilon(x-y)\, dt.
\end{equs}
In the $\delta\to0$ limit this yields
\begin{equs}          
&\frac{1}{h}\E \int_0^h\int_{x,y} |u(t,x)-\xi(y)|^2  \varrho_\varepsilon(x-y)\,dt 
\\
&\leq 2\E \int_{x,y}|\xi(x)-\xi(y)|^2\varrho_\varepsilon(x-y) \, dx
\\
& +\frac{N}{\varepsilon^2}\E \int_0^{2h}\int_{x} (|u(t,x)|^{m+1}+|\xi(x)|^{m+1})\, dt
\\
&+2\E \int_0^{2h} \int_{x,y}|\sigma(x,u(t,x))|_{l_2}^2 \varrho_\varepsilon(x-y)\, dt,
\end{equs}
which implies that 
\begin{equs}
\limsup_{h \to 0}\frac{1}{h}\E &\int_0^h\int_{x,y} |u(t,x)-\xi(y)|^2  \varrho_\varepsilon(x-y)\,dt 
\\
&\leq 2\E \int_{x,y}|\xi(x)-\xi(y)|^2\varrho_\varepsilon(x-y).
\end{equs}
Consequently, by \eqref{eq:spliting} we get
\begin{align*}
&\limsup_{h \to 0}\frac{1}{h}\E \int_0^h\int_{x} |u(t,x)-\xi(x)|^2  \,dt 
\leq 3\E \int_{x,y}|\xi(x)-\xi(y)|^2\varrho_\varepsilon(x-y),
\end{align*}
from which the claim follows, since right hand side goes to $0$ as $\varepsilon \to 0$ due to continuity of translations in $L_2(\bT^d)$.
\end{proof}

\subsection{The $(\star)$-property}
We now introduce a somewhat loaded, but important definition.
First take $ g \in C^\infty(\bR)$ with $g' \in C^\infty_c(\bR)$, $\varrho \in C^\infty(\bT^d \times \bT^d)$, $\varphi \in C^\infty_c((0,T))$,
$\tsigma\in C(\T^d\times\R)$ with linear growth,
and $\tu\in L_{m+1}(\Omega_T; L_{m+1}( \T^d))$.
We then introduce the notations, for $\theta>0$, $a\in\R$,
\begin{equ}
\phi_\theta(t,x,s,y):=\varrho(x,y) \rho_\theta(t-s) \varphi\left(\frac{t+s}{2}\right),
\end{equ}
\begin{equ}
F_\theta(t,x,a):= \int_0^T \int_{y} \tsigma^k(y, \tu(s,y)) g(\tu(s,y)-a)\phi_\theta(t,x,s,y) \,d\beta^k(s).
\end{equ}
\begin{remark}
There exists a version of the function $F_\theta$ which is smooth in $(t,x,a)$ (see, e.g., \cite[Exercise 3.15, page 78]{Kun}). We will always use this version. 
\end{remark}
Set $\mu=\mu(m)=\frac{3m+5}{4(m+1)}$, which is chosen so that one has $\frac{m+3}{2(m+1)}<\mu<1$.
\begin{definition}\label{def:star}
A function $u\in L_{m+1}(\Omega\times Q_T)$ is said to have the $(\star)$-property if for all $g,\varrho,\varphi,\tsigma,\tu$ as above,
and for all sufficiently small $\theta>0$, we have that $F_\theta(\cdot, \cdot, u) \in L_1(\Omega_T \times \bT^d)$ and

\begin{align}
\nonumber
&\E \int_{t,x} F_\theta(t,x, u(t,x)) \leq  N\theta^{1-\mu}
\\  \label{eq:strong-entropy}
& -  \E \int_{s,t,x,y} \sigma^k(x, u(s,x)) \tsigma^k(y, \tilde{u}(s,y))g'(u(s,x)-\tilde{u}(s,y))\phi_{\theta}(t,x,s,y),
\end{align}
with some constant $N$ independent of $\theta$.
\end{definition}

Note that this property is only related to \eqref{eq:main_equation} through the stochastic part $\sigma$.
If the choice of $\sigma$ needs to be emphasized, we talk about the $(\star)$-property with coefficient $\sigma$.
In the terminology of \cite{FENG} and  \cite{Biswas}, an entropy solution with the $(\star)$-property can be referred to as a strong entropy solution.

Let us first derive some basic estimates for $F_\theta$, and hence we fix  $g, \varrho, \varphi,\tsigma, \tu$ as above.
First note that since $\varphi$
is compactly supported in $(0,T)$, and $\rho_\theta(t-\cdot)$ is supported in $[t-\theta, t]$,
we have for $\theta $ sufficiently small,
\begin{equ}\label{eq:F rewrite}
F_\theta (t,x,a)= I_{t> \theta} \int_{t-\theta}^t \int_{y} \sigma^k(y, \tu(s,y)) g(\tu(s,y)-a)\phi_\theta(t,x,s,y) \, d\beta^k(s).
\end{equ}
\begin{lemma}\label{lem:F}
For any  $\lambda\in ( \frac{m+3}{2(m+1)}, 1)$, $k\in\mathbb{N}$ we have
\begin{equation}\label{eq:F estimate}
\E \| \D_ a F_\theta \|_{L_\infty([0,T]; W^{k}_{m+1}(\bT^d \times \bR))}^{m+1} \leq N \theta^{-\lambda (m+1)}(1+
\E \| \tu\|_{L_{m+1}(Q_T)}^{m+1}),
\end{equation} 
where $N$ depends only on $k,p,d, T ,\lambda$,  and the functions $g, \varrho, \varphi,\tsigma, \tu$, but not on $\theta$.
\end{lemma}
\begin{proof}
To ease the notation we suppress the $y\in\T^d$ argument in $\tsigma$ and the $s,y\in Q_T$ arguments in $\tu$.
For any  $q \in \mathbb{N}^d$, $l \in \mathbb{N}$, $j \in \{0,1\}$, we have 
by the Burkholder-Gundy-Davis inequality
\begin{equs}
 \E  |&\D_t^j\D^{l+1}_a \D^q_xF_\theta (t,x,a ) |^{m+1} \\
 \le & \E  I_{t>\theta}\left[\int_{t-\theta}^t \sum_k\left(\int_{y} \tsigma^k(\tu) \D_a^{l+1} g(\tu-a) \D^q_x \D^j_t \phi_\theta(t,x,s,y) \right)^2  \, ds \right]^{(m+1)/2}
\\
 \le &  \E I_{t>\theta}\left[\int_{t-\theta}^t \| \tsigma(\tu)\|_{\ell_2(L_2(\bT^d))}^2 
\| \D_a^{l+1}g(\tu-a)\|_{L_2(\bT^d)}^2
 \theta^{-2(1+j)} \, ds \right]^{(m+1)/2}
 \\
 \le & \theta^{\frac{m-1}{2}}\theta^{-(m+1)(1+j)}  \E I_{t>\theta}\left[\int_{t-\theta}^t \| \tsigma(\tu)\|_{\ell_2(L_2(\bT^d))}^{m+1} 
\| \D_a^{l+1}g(\tu-a)\|_{L_2(\bT^d)}^{m+1}
 \, ds \right].
 \label{eq:choose-j}
\end{equs}
Choosing $j=0$,  summing  over all $|q|+l \leq k$,  integrating over $[0,T] \times \T^d \times \bR$,  and using the fact that $g' \in C^\infty_c(\bR)$ and the trivial estimate \eqref{eq:whole-theta}, we obtain
$$
\E \|\D_a F_\theta \|_{L_{m+1}([0,T]; W^{k}_{m+1}(\bT^d \times \bR))}^{m+1}\le \theta^{\frac{m-1}{2}}\theta^{-m} \E \int_t \| \tsigma(\tu)\|_{\ell_2(L_2(Q_T))}^{m+1},
$$
which by the linear growth of $\tsigma$ gives 
\begin{equation}               \label{eq:LpWk}
\E \|\D_a F_\theta \|_{L_{m+1}([0,T]; W^{k}_{m+1}(\bT^d \times \bR))}^{m+1} \le \theta^{-\frac{m+1}{2}} (1+ \E \| \tu\|_{L_{m+1}(Q_T)}^{m+1}).
\end{equation} 
Similarly, choosing $j=1$ in \eqref{eq:choose-j}, summing  over all $|q|+l \leq k$, and integrating over $[0,T] \times \bT^d \times \bR$ gives
\begin{equation}                   \label{eq:W1Wk}
\E \|\D_a F_\theta \|_{W^1_{m+1}([0,T] ; W^{k}_{m+1}(\bT^d \times \bR))}^{m+1} \le  \theta^{-3\frac{(m+1)}{2}} (1+ \E \| \tu\|_{L_{m+1}(Q_T)}^{m+1}).
\end{equation}
By interpolating between \eqref{eq:LpWk} and \eqref{eq:W1Wk} (see also \cite[Section 1.18.4]{Triebel} for the treatment of the extra $L_{m+1}(\Omega)$ space)  we have for $\delta \in [0,1]$
\begin{equation}
\E \|\D_a F_\theta \|_{W^\delta_{m+1}([0,T] ; W^{k}_{m+1}(\bT^d \times \bR))}^{m+1} \le  \theta^{-(m+1)(1+2\delta)/2} (1+ \E \| \tu\|_{L_{m+1}(Q_T)}^{m+1}).
\end{equation}
For arbitrary $\delta \in (1/(m+1),1/2)$, we set  $\lambda=(1+2\delta)/2$, 
and the claim follows by Sobolev embedding.
\end{proof}
\begin{remark}
Since $k$ was arbitrary, by Sobolev embedding one can estimate the $L_\infty(Q_T\times\R)$ norm of $\partial_\alpha \D_aF_\theta$ by the right-hand side of \eqref{eq:F estimate}, for any multi-index $\alpha$ in the variables $x,a$. We will do so in the sequel, in fact always with the choice $\lambda=\mu$.
\end{remark}

\begin{corollary}\label{cor} 
(i) Let $u_n$ be a sequence bounded in $L_{m+1}(\Omega_T\times \bT^d)$, satisfying the $(\star)$-property with coefficient $\sigma_n$, uniformly in $n$. 
Suppose that $u_n$ converges for almost all $\omega,t,x$ to a function $u$ and $\sigma_n$ converges in the supremum norm to $\sigma$. Then $u$ has the $(\star)$-property with coefficient $\sigma$.

(ii) Let $u\in L_2(\Omega\times Q_T)$. Then one has for all $\theta>0$
\begin{equ}\label{eq:0lambda limit}
\E \int_{t,x} F_\theta(t,x, u(t,x)) = \lim_{\lambda \to 0}  \E \int_{t,x,a} F_\theta(t,x, a) \rho_\lambda(u(t,x)-a) \, .
\end{equ}
\end{corollary}

\begin{proof}
(i) We have that $\lim_{n \to \infty} F_\theta(t,x, u_n(t,x))=F_\theta(t,x, u(t,x))$ for almost all $(\omega,t,x)$. Moreover, 
\begin{equ}\label{eq:simple}
|F_\theta(t,x, u_n(t,x))|\leq \|\D_aF_\theta\|_{L_\infty(Q_T\times\R)} |u_n(t,x)|+|F(t,x,0)|. 
\end{equ}
By Lemma \ref{lem:F}, and the fact that  $\E\int_{t,x}|F_\theta(t,x,0)|<\infty$,
we see that the right hand side above is uniformly integrable in $(\omega,t,x)$.
Hence, one can take limits on
the left-hand side of \eqref{eq:strong-entropy} to get
$$
\lim_{n\to\infty}\E \int_{t,x} F_\theta(t,x, u_n(t,x))=\E \int_{t,x} F_\theta(t,x, u(t,x)).
$$
By similar (in fact, easier) arguments one can see the convergence of the second term on the right-hand side of \eqref{eq:strong-entropy}, and since   the constant $N$ was assumed to be independent of $n \in \bN$, we get the claim.

(ii) Writing
$$
\big| F_\theta(t,x,u(t,x))-\int_aF_\theta(t,x, a) \rho_\lambda(u(t,x)-a) \,  \big|
\leq \lambda\| \D_a F_\theta\|_{L_\infty(Q_T\times\R)},
$$
the claim simply follows from Lemma \ref{lem:F}.
\end{proof}

Finally we formulate a simple technical statement used in the proof of Theorem \ref{thm:uniqueness} below.
\begin{proposition}\label{prop:theta conv}
Let $v\in L_{1}(\Omega\times Q_T)$, and let
$\bar{\rho}_\theta(\cdot)$ be either $\rho_\theta(\cdot)$ or $\rho_\theta(-\cdot)$.
Furthermore, let $g\in C_c^\infty\big(((0,T)\times\T^d)^2\big)$,
$f\in L_1(\Omega\times Q_T)$,
and let $h:\Omega\times Q_T\times \R\rightarrow\R$ be a function bounded by $C$,
such that $|h(t,x,a)-h(t,x,b)|\leq C |a-b|$, for some constant $C$.
Then
\begin{equs}
\E\int_{t,s,x,y}&h(t,x,v(s,y))f(t,x)g(t,x,s,y)\bar{\rho}_\theta(t-s)
\\
&\underset{\theta\rightarrow0}{\rightarrow}
\E\int_{t,x,y}h(t,x,v(t,y))f(t,x)g(t,x,t,y).
\end{equs}
\end{proposition}
\begin{proof}
First notice that since in the variables $s,t$ the function $g$ is supported compactly in $(0,T)^2$ and $\bar{\rho}_\theta$ is supported in $[-\theta,\theta]$, the integration over $[0,T]^2$ can be replaced by integration over $\bR^2$ and $f, v$ can be set equal to zero on the complement of $[0,T]$. 
First, one has the bound
\begin{equs}
\big|\E&\int_{\bR^2}\int_{x,y}h(t,x,v(s,y))f(t,x)\big(g(t,x,s,y)-g(t,x,t,y)\big)\bar{\rho}_\theta(t-s)\,ds\,dt\big|\leq N\theta,
\end{equs}
where here and below $N$ is some constant depending on the data $C,T$, and the norms of $g,f$.
One also has
\begin{align}
\nonumber
\big|\E&\int_{\bR^2}\int_{x,y}\big(h(t,x,v(s,y))-h(t,x,v(t,y))\big)f(t,x)g(t,x,t,y)\bar \rho_\theta(t-s)\,ds\,dt\big|
\\    \nonumber
&\leq 
N
\E\int_{\bR^2}\int_{y}\big(|v(t,y)-v(s,y)|\wedge 1 \big)f(t,x)\bar \rho_\theta(t-s)\,ds\,dt
\\   
&\leq  N
\E\int_\R e_\theta(t)F(t)\,dt,
\label{eq:theta-convergence}
\end{align}
where the (random) processes $e_\theta$ and $F$ are defined by
\begin{equ}
e_\theta(t)= \int_{\bR}\int_{y}(|v(t,y)-v(s,y)|\wedge 1)\bar \rho_\theta(t-s)\,ds,
\quad
F(t)=\int_x|f(t,x)|.
\end{equ}
Since $t\mapsto v(t,\cdot)$ belongs to $L_{1}(\R,L_{1}(\Omega \times \bT^d))$,
by the continuity of translations we get that
 $e_\theta\to 0$ 
as $\theta\rightarrow0$ in $L_1(\Omega\times \R)$, and hence also
in measure on $\Omega \times \R$. 
Since $e_\theta$ is also uniformly bounded by $1$ and $F\in L_1(\Omega\times \R)$, it follows that 
the right-hand side of \eqref{eq:theta-convergence}
 tends to zero as $\theta \to 0$.  This finishes the proof.
\end{proof}

\section{Generalised $L_1$-contraction}
The following result is the main step to obtain Theorem \ref{thm:main}. Part (i) is the $L_1$-contraction similar to our main result \eqref{eq:main contraction}, under some additional assumptions.
Part (ii) is a generalised variant of the $L_1$ contraction, where we may allow the equation itself to vary and thus deduce stability properties,
at the cost of only controlling the time integral of the $L_1$-norm of the solution.

\begin{theorem}  \label{thm:uniqueness}
Let $(A, \xi)$,  $(\tA, \tilde{\xi})$   satisfy Assumption \ref{as:A}, and $\sigma,\tsigma$ satisfy Assumption \ref{as:noise} \eqref{as:sigma}.
Let $u$ and $\tu$ be two entropy solutions of $\cE(A,\sigma,\xi)$ and $\cE(\tA,\tsigma, \txi)$, respectively,
and assume that $u$ has the $(\star)$-property with coefficient $\sigma$.
Then,

(i) if furthermore $A=\tA$ and $\sigma=\tsigma$, then 
\begin{equ}\label{eq:L_1 contraction}
\esssup_{t\in[0,T]}\E\int_x|u(t,x)-\tu(t,x)|\leq\E\int_x|\xi(x)-\txi(x)|.
\end{equ}

(ii) for all $\eps,\delta\in(0,1]$, $\lambda\in[0,1]$ and $\alpha\in(0,1\wedge(m/2))$,  we have 
\begin{equs}
\E \int_{t,x} & |u(t,x)-\tilde{u}(t,x)|  
\leq T\E \int_{x}|\xi(x)-\txi(x)|
\\
&+N\eps^{\frac{2}{m+1}}\big(1+\E\|\nabla\Psi(u)\|_{L_1(Q_T)}\big)
\\
&+T\sup_{|h|\leq\eps}\big(\E\|\txi(\cdot)-\txi(\cdot+h)\|_{L_1(\T^d)}\big)
\\
&+N\eps^{-2}\E\big(\|I_{|u|\geq R_\lambda}(1+|u|)\|_{L_m(Q_T)}^m+
\|I_{|\tu|\geq R_\lambda}(1+|\tu|)\|_{L_m(Q_T)}^m\big)
\\
&+N\Big( \delta^{2\kappa}+ \varepsilon^{2\bar\kappa} \delta^{-1}+ \varepsilon^{-2} \delta^{2\alpha}+\eps^{-2}\lambda^2+\red{\delta^{-1}\sup_{x,r} \frac{|\sigma(x,r)-\tsigma(x,r)|^2_{l_2}}{(1+|r|)^{-(m+1)}} }
\Big)
\\
&\quad\quad\times\E(1+\|u\|^{\red{m+1}}_{L_{\red{m+1}}(Q_T)}+\|\tu\|_{L_{\red{m+1}}(Q_T)}^{\red{m+1}}),
\label{eq:super inequality}
\end{equs}
where $N$ depends only on $m,K,d,T,\alpha$, and $R_\lambda$ is given by
\begin{equ}\label{eq:R lambda}
R_\lambda=\sup\{R\in[0,\infty]:\,|\fra(r)-\tfra(r)|\leq\lambda, \,\,\forall |r|<R\}.
\end{equ}
\end{theorem}

\begin{remark}
Note that all the norms of the solutions on the right-hand side of \eqref{eq:super inequality} are finite by Definition \ref{def:solution}.
\end{remark}

\begin{proof}
The majority of the proof is identical for (i) and (ii), so their separation is postponed to the very end.
Denote $\varrho_\eps=\rho_\eps^{\otimes d}$, and fix a $\varphi \in C^\infty_c((0,T))$ such that 
$\|\varphi\|_{L_\infty([0,T])}\vee\|\partial_t\varphi\|_{L_1([0,T])}\leq 1$.
Introduce, for $\theta,\eps>0$,
\begin{equ}
\phi_{\theta, \varepsilon}(t, x,s,y)= \rho_{\theta}(t-s)\varrho_\varepsilon\left(x-y\right) \varphi \left(\tfrac{t+s}{2}\right),
\,\,
\phi_{\eps}(t,x,y)=\varrho_\eps(x-y)\varphi(t).
\end{equ}
Furthermore,  for each $\delta>0$, let $\eta_\delta\in C^2(\R)$ be defined by
\begin{equ}
\eta_\delta(0)=\eta_\delta^\prime(0)=0,\quad\eta_\delta^{\prime\prime}(r)=\rho_\delta(|r|).
\end{equ}
Note that
\begin{equ}\label{eq:eta prop}
\big|\eta_\delta(r)-|r|\big|\leq\delta,\quad\supp\eta_\delta^{\prime\prime}\subset[-\delta,\delta],
\quad\int_\R|\eta_\delta^{\prime\prime} (r-\zeta)|\,d\zeta\leq2,
\quad|\eta^{\prime\prime}_\delta|\leq 2\delta^{-1}.
\end{equ}
We apply the entropy inequality \eqref{eq:entropy_inequality} with $\eta_\delta(\cdot-a)$
in place of $\eta$ and $\phi_{\theta,\eps}(\cdot,\cdot,s,y)$ in place of $\phi$,
for some $s\in[0,T]$, $y\in\T^d$, $a\in\R$.
Assuming that $\theta$ is sufficiently small, one has $\phi_{\theta, \varepsilon}(0, x,s,y)=0$,
and thus we get
\begin{equs}              
&-\int_{t,x}\eta_\delta(u(t,x)-a) \D_t\phi_{\theta, \varepsilon}(t,x,s,y) 
 \leq 
\int_{t,x}q_\delta(u(t,x),a) \Delta_x \phi_{\theta, \varepsilon}(t,x,s,y)  
\\
&\quad\quad-\int_{t,x} \eta_\delta^{\prime\prime}(u(t,x)-a) | \nabla_x \Psi(u(t,x))|^2\phi_{\theta, \varepsilon}(t,x,s,y)  
\\
&\quad\quad+\frac{1}{2}\int_{t,x} \eta_\delta^{\prime \prime}(u(t,x)-a)|\sigma(x,u(t,x))|_{l_2}^2 \phi_{\theta, \varepsilon}(t,x,s,y)
\\   \label{eq:ineq1}
&\quad\quad+ \int_{0}^T\int_{x}\eta_\delta^{\prime}(u(t,x)-a)\sigma^k(x,u(t,x))\phi_{\theta, \varepsilon}(t,x,s,y) \, d\beta^k(t),
\end{equs}
where 
$$
q_\delta(r,a)=\int_a^r \eta_\delta'(\zeta -a)\fra^2(\zeta) \, d\zeta.
$$
Notice that all the expressions in \eqref{eq:ineq1} are continuous in $(a,s,y)$.
We now substitute $a= \tilde{u}(s,y)$, integrate over $(s,y)$, and take expectations.
For the last term in \eqref{eq:ineq1} this is justified by \eqref{eq:simple}. All of the other terms are continuous in $a$ and can be bounded by $N(|a|^m+X)$ with some constant $N$ and some integrable random variable $X$ (recall \eqref{eq:as fra}), so that  substituting $a= \tilde{u}(s,y)$ and integrating out $s$, $y$, and $\omega$, results in finite quantities.

After writing the analogous inequality with the roles of $u,t,x$ and $\tu,s,y$ reversed,  using the symmetry of $\eta_\delta$, and adding both inequalities, one arrives at
\begin{equs}                
-\E&\int_{t,x,s,y}\eta_\delta(u(t,x)-\tilde{u}(s,y)) (\D_t\phi_{\theta, \varepsilon}(t,x,s,y)+ \D_s\phi_{\theta, \varepsilon}(t,x,s,y))
\\
&\leq\E \int_{t,x,s,y}q_\delta(u(t,x),\tilde{u}(s,y)) \Delta_x \phi_{\theta, \varepsilon}(t,x,s,y) 
\\ 
&\quad+\E \int_{t,x,s,y}\tq_\delta(\tilde{u}(s,y),u(t,x)) \Delta_y \phi_{\theta, \varepsilon}(t,x,s,y)
\\ 
&\quad-\E \int_{t,x,s,y}\eta_\delta^{\prime\prime}(u(t,x)-\tilde{u}(s,y)) | \nabla_x \Psi(u(t,x))|^2\phi_{\theta, \varepsilon}(t,x,s,y)  
\\
&\quad-\E \int_{t,x,s,y}\eta_\delta^{\prime\prime}(\tu(s,y)-u(t,x)) | \nabla_y \tPsi(\tilde{u}(s,y))|^2\phi_{\theta, \varepsilon}(t,x,s,y)  
\\
&\quad+\E \frac{1}{2}\int_{t,x,s,y} \eta_\delta^{\prime\prime}(u(t,x)-\tilde{u}(s,y))|\sigma(x,u(t,x))|^2 \phi_{\theta, \varepsilon}(t,x,s,y)
\\
&\quad+\E \frac{1}{2}\int_{t,x,s,y} \eta_\delta^{\prime\prime}(\tilde{u}(s,y)-u(t,x))|\tsigma(y,\tilde{u}(s,y))|^2 \phi_{\theta, \varepsilon}(t,x,s,y)
\\   
&\quad+ \E \int_{y,s} \left[\int_0^T \int_{x}\eta_\delta^{\prime}(u(t,x)-a)\sigma^k(x,u(t,x))\phi_{\theta, \varepsilon}(t,x,s,y) \, d\beta^k(t)\right]_{a=\tilde{u}(s,y)} 
\\
\label{eq:everything}
&\quad+ \E \int_{t,x} \left[\int_0^T \int_{y}\eta_\delta^{\prime}(\tilde{u}(s,y)-a)\tsigma^k(y,\tilde{u}(s,y))\phi_{\theta, \varepsilon}(t,x,s,y) \,  d\beta^k(s)\right]_{a=u(t,x)} 
\\
&=:\sum_{i=1}^8 B_i.
\end{equs}
Notice that one has
\begin{equ}
\D_t\phi_{\theta, \varepsilon}(t,x,s,y)+ \D_s\phi_{\theta, \varepsilon}(t,x,s,y)=
\rho_\theta(t-s)\varrho_\eps(x-y)(\partial_t\varphi)\left(\tfrac{t+s}{2}\right).
\end{equ}
We now pass to the $\theta\to0$ limit.
For the left-hand side, thanks to the above identity, we may apply
Proposition \ref{prop:theta conv} twice: first, with
\begin{equs}
v&=\tu,&\quad h(t,x,a)&=\frac{\eta_\delta(u(t,x)-a)}{1\vee|u(t,x)|+1\vee |a|},
\\
f&=1\vee |u|,&\quad g(t,x,s,y)&=\varrho_\eps(x-y)(\partial_t\varphi)\left(\tfrac{t+s}{2}\right)
\end{equs}
and second, with the with the roles of $u,t,x$ and $\tu,s,y$ reversed. For $B_1$, we set the roles as
\begin{equs}
v&=\tu,&\quad h(t,x,a)&=\frac{\int_0^{u(t,x)} \eta_\delta'(\zeta -a)\fra^2(\zeta) \, d\zeta}{1\vee|u(t,x)|^m},
\\
f&=1\vee |u|^m,&\quad g(t,x,s,y)&=\Delta_x \phi_{\varepsilon}(\tfrac{t+s}{2},x,y),
\end{equs}
and for $B_2$ we symmetrise as above. For $B_3$ we set
\begin{equ}
v=\tu,\quad h(t,x,a)=\eta_\delta^{\prime\prime}(u(t,x)-a),\quad f=|\nabla\Psi(u)|^2,\quad
g(t,x,s,y)=\phi_{\varepsilon}(\tfrac{t+s}{2},x,y),
\end{equ}
and symmetrically for $B_4$. For $B_5$ we simply change $|\nabla\Psi(u)|^2$ to $|\sigma(u)|^2$ above, and symmetrically for $B_6$.

Next note that $B_7=0$: an $\cF_s$-measurable quantity is substituted in a stochastic integral from $s$ to $s+\theta$, so after taking expectation the term vanishes (to make this argument completely precise, we refer to \eqref{eq:0lambda limit} and \eqref{eq:F rewrite}).
Finally, for $B_8$, we use the $(\star)$-property of $u$: in the notation of Definition \ref{def:star} we choose
$\eta_\delta^\prime$ in place of $g$ and $\varrho_\eps(x-y)$ in place of $\varrho(x,y)$.
After sending $\theta$ to $0$, we thereby obtain
\begin{equs}              
-\E&\int_{t,x,y}\eta_\delta(u(t,x)-\tilde{u}(t,y))\D_t \phi_{ \varepsilon}(t,x,y) 
\\
&\leq\E \int_{t,x,y}q_\delta(u(t,x),\tilde{u}(t,y)) \Delta_x\phi_{ \varepsilon}(t,x,y) 
\\ 
 &\quad+\E \int_{t,x,y}\tq_\delta(\tilde{u}(t,y),u(t,x)) \Delta_y\phi_{ \varepsilon}(t,x,y)
\\ 
&\quad-\E \int_{t,x,y}\eta_\delta^{\prime\prime}(u(t,x)-\tilde{u}(t,y)) | \nabla_x \Psi(u(t,x))|^2\phi_{ \varepsilon}(t,x,y) 
\\  
&\quad-\E \int_{t,x,y}\eta_\delta^{\prime\prime}(u(t,x)-\tilde{u}(t,y)) | \nabla_y \tPsi(\tilde{u}(t,y))|^2\phi_{ \varepsilon}(t,x,y) 
\\
&\quad+\E \frac{1}{2}\int_{t,x,y}\eta_\delta^{\prime\prime}(u(t,x)-\tilde{u}(t,y))|\sigma(x,u(t,x))|_{l_2}^2 \phi_{ \varepsilon}(t,x,y)
\\
&\quad+\E \frac{1}{2}\int_{t,x,y} \eta_\delta^{\prime\prime}(u(t,x)-\tilde{u}(t,y))|\tsigma(y,\tilde{u}(t,y))|_{l_2}^2\phi_{ \varepsilon}(t,x,y)
\\ 
& \quad-  \E \int_{t,x,y}\sigma(x, u(t,x)) \tsigma(y, \tilde{u}(t,y))\eta_\delta^{\prime\prime}(u(t,x)-\tilde{u}(t,y))\phi_{\varepsilon}(t,x,y) 
\\  \label{eq:no-theta}
&=: \sum_{i=1}^7 C_i   .
\end{equs}
By the second and fourth property in \eqref{eq:eta prop}, for the last three terms we can write, using Assumption \ref{as:noise} \eqref{as:sigma}
\begin{equs}
&C_5+C_6+C_7
\\
&=\frac{1}{2}\E \int_{t,x,y} \eta_\delta^{\prime\prime}(u(t,x)-\tilde{u}(t,y))|\sigma(x,u(t,x))-\tsigma(y,\tilde{u}(t,y))|_{l_2}^2\phi_{ \varepsilon} (t,x,y)
\\
&\le   \E \int_{t,x,y} \eta_\delta^{\prime\prime}(u(t,x)-\tilde{u}(t,y))|\sigma(x,u(t,x))-\sigma(x,\tilde{u}(t,y))|_{l_2}^2\phi_{ \varepsilon} (t,x,y)
\\
&\quad+ \E \int_{t,x,y} \eta_\delta^{\prime\prime}(u(t,x)-\tilde{u}(t,y))|\sigma(x,\tilde{u}(t,y))-\sigma(y,\tilde{u}(t,y))|_{l_2}^2\phi_{ \varepsilon}(t,x,y)
\\
&\quad+\E \int_{t,x,y} \eta_\delta^{\prime\prime}(u(t,x)-\tilde{u}(t,y))|\sigma(y,\tilde{u}(t,y))-\tsigma(y,\tilde{u}(t,y))|_{l_2}^2\phi_{ \varepsilon}(t,x,y)
\\
& \le (  \delta^{2\kappa}+ \varepsilon^{2\bar\kappa}\delta^{-1}+\delta^{-1}\red{ \sup_{r,x}\frac{|\sigma(x,r)-\tsigma(x,r)|^2_{l_2}}{(1+|r|)^{-(m+1)}}} ) 
\\
& \times \red{\E(1+\|u\|_{L_{m+1}(Q_T)}^{m+1}+\|\tu\|_{L_{m+1}(Q_T)}^{m+1})}. 
\\
\label{eq:stochastic-terms}
\end{equs}
We emphasize for later use that the bound $|\eta^{\prime\prime}_\delta|\leq 2\delta^{-1}$ is only used in the above estimate.

From now on, if confusion does not arise,  we drop the arguments $t,x,y$, keeping in mind that $u$ is a function of $(t,x)$, $\tilde{u}$ is a function of $(t,y)$, and $\phi_\eps$ is a function of $(t,x,y)$.
By the relation $\D_{x_i}\phi_\varepsilon= - \D_{y_i} \phi_\varepsilon$ we have 
\begin{equs}
C_1& =\E \int_{t,x,y} q_\delta(u,\tilde{u})\D_{x_i}^2 \phi_{\varepsilon}=-\E \int_{t,x,y}   q_\delta(u,\tilde{u})\D^2_{x_iy_i}\phi_\eps,
\end{equs}
and hence
\begin{equs}
C_1=&-\E\int_{t,x,y}\D^2_{x_iy_i}\phi_\eps\int_\tu^u\eta^\prime_\delta(\zeta-\tu)\fra^2(\zeta)\,d\zeta
\\
=&-\E\int_{t,x,y}\D^2_{x_iy_i}\phi_\eps\int_\tu^u\int_\tu^\zeta\eta^{\prime\prime}_\delta(\zeta-r)\fra^2(\zeta)\,dr\,d\zeta
\\
=&  -\E\int_{\tu \leq u} \D^2_{x_iy_i}\phi_\eps\int_\tu^u\int_\tu^uI_{r \leq \zeta}\eta^{\prime\prime}_\delta(\zeta-r)\fra^2(\zeta)\,dr\,d\zeta
\\
&-\E\int_{\tu \geq u} \D^2_{x_iy_i}\phi_\eps\int_u^\tu\int_u^\tu I_{r \geq \zeta}\eta^{\prime\prime}_\delta(\zeta-r)\fra^2(\zeta)\,dr\,d\zeta.
\label{eq:C1}
\end{equs}
Symmetrically, one has
\begin{equs}
C_2=&-\E\int_{t,x,y}\D^2_{x_iy_i}\phi_\eps\int_u^\tu\int_u^\zeta\eta^{\prime\prime}_\delta(\zeta-r)\tfra^2(\zeta)\,dr\,d\zeta
\\
=&  -\E\int_{\tu \leq u} \D^2_{x_iy_i}\phi_\eps\int_\tu^u\int_\tu^u I_{r \geq \zeta}\eta^{\prime\prime}_\delta(\zeta-r)\tfra^2(\zeta)\,dr\,d\zeta
\\
&-\E\int_{\tu \geq u} \D^2_{x_iy_i}\phi_\eps\int_u^\tu \int_u^\tu I_{r \leq \zeta}\eta^{\prime\prime}_\delta(\zeta-r)\tfra^2(\zeta)\,dr\,d\zeta.
\label{eq:C2}
\end{equs}
Notice also that by \eqref{item:chain_ruleW2}  of Definition \ref{def:solution} and Remark \ref{rem:W1}  we have
\begin{equs}    
C_3+C_4 \leq& -2\E \int_{t,x,y}  \eta_\delta^{\prime\prime}(u-\tilde{u})\nabla_x\Psi(u) \cdot\nabla_y \tPsi(\tilde{u}) \phi_{ \varepsilon}
\\
\nonumber
=&-2\E \int_{t,x,y} \phi_\varepsilon \D_{x_i} \Psi(u) \D_{y_i} \int_u^{\tilde{u}} \eta_\delta^{\prime\prime} (u-r)\tfra(r) \, dr 
\\
\nonumber
=&2\E \int_{t,x,y} \D_{y_i}\phi_\varepsilon \D_{x_i} \Psi(u)  \int_u^{\tilde{u}} \eta_\delta^{\prime\prime} (u-r)\tfra(r) \, dr 
\\
\nonumber
=&-2\E \int_{t,x,y} \D^2_{x_iy_i}\phi_\varepsilon \int_{\tilde{u}}^u  \int_\zeta^{\tilde{u}} \eta_\delta^{\prime\prime} (\zeta-r)\tfra(r) \, dr \fra(\zeta)\, d\zeta
\\
=&  2\E\int_{\tu \leq u} \D^2_{x_iy_i}\phi_\eps\int_\tu^u\int_\tu^u I_{r \leq \zeta}\eta^{\prime\prime}_\delta(\zeta-r)\tfra(r) \fra(\zeta) \,dr\,d\zeta
\\
&+2\E\int_{\tu \geq u} \D^2_{x_iy_i}\phi_\eps\int_u^\tu\int_u^\tu I_{r \geq \zeta}\eta^{\prime\prime}_\delta(\zeta-r)\tfra(r) \fra(\zeta)\,dr\,d\zeta.      \label{eq:C3C4}
\end{equs}
Relabelling the variables $r\leftrightarrow \zeta$ in \eqref{eq:C2} (recall that $\eta''_\delta$ is even) and adding \eqref{eq:C1}, \eqref{eq:C2}, and \eqref{eq:C3C4} we obtain 
\begin{equation}                 \label{eq:C1+C4}
C_1+C_2+C_3+C_4 \leq \E\int_{t,x,y} |\D^2_{x_iy_i}\phi_\eps|\int_\tu^u\int_\tu^u \eta^{\prime\prime}_\delta(\zeta-r)|\fra(r)-\tfra(\zeta)|^2 \,dr\,d\zeta.
\end{equation}
Let us set 
\begin{equs}
D:=\int_{\tilde{u}}^u  \int_{\tilde{u}}^u \eta_\delta^{\prime\prime} (\zeta-r)
|\fra(r)-\tfra(\zeta)|^2\, dr \,d\zeta.
\end{equs}
We easily see that (recall
the third property in \eqref{eq:eta prop})
\begin{equs}
|D|&\le D_1+D_2\label{eq:D}
\\
&:=\int_{\tilde{u}}^u  \int_{\tilde{u}}^u \eta_\delta^{\prime\prime} (\zeta-r)
|\fra(\zeta)-\fra(r)|^2\, dr \,d\zeta + 
\int_{\tilde{u}\wedge u}^{u \vee \tu} 
|\fra(\zeta)-\tfra(\zeta)|^2\,d\zeta.
\end{equs}
At this point we make use of Assumption \ref{as:A}, to get the bounds
\begin{equ}
I_{|r-\zeta| \leq \delta, |\zeta|\geq 2 \delta }\left|\fra(r)-\fra(\zeta)\right|
\leq \delta I_{|\zeta|\geq 2 \delta}\sup_{r\in[\zeta-\delta,\zeta+\delta]}|\fra'(r)|
\label{eq:with-delta}             
 \le \delta |\zeta|^{\frac{m-3}{2}}, 
\end{equ}
and 
\begin{equs}              
I_{|r-\zeta| \leq \delta}\left|\fra(r)-\fra(\zeta)\right| &\le
 I_{|r-\zeta|\leq \delta}( \fra(r)+\fra(\zeta))
 \\
&\le
\int_0^{2(|\zeta|\vee \delta)}s^{\frac{m-3}{2}}\,ds
\le  (|\zeta|\vee\delta)^{\frac{m-1}{2}}.  \label{eq:without-delta}
\end{equs}
Combining \eqref{eq:with-delta} and \eqref{eq:without-delta} we get (recall that $\alpha\in(0,1)$) 
\begin{equ}
I_{|r-\zeta| \leq \delta, |\zeta|\geq 2 \delta }\left|\fra(r)-\fra(\zeta)\right|
 \le\delta^\alpha |\zeta|^{{\frac{m-1}{2}}-\alpha}.
\end{equ}
Using again \eqref{eq:eta prop},
and that by assumption $m-1-2\alpha>-1$, we have
\begin{equs}
D_1
&\le\int_{\tilde{u}\wedge u }^{\tu\vee u}  ( I_{|\zeta|\leq2 \delta} \delta^{m-1}+I_{|\zeta|\geq2 \delta}\delta^{2\alpha}\zeta^{m-1-2\alpha}) \,d\zeta
\\
& \le \big(\delta^m + \delta^{2\alpha}(|u|+ |\tilde{u}|)^{m-2\alpha}  \big)
\le \delta^{2\alpha} \big(1+|u|^{m}+|\tu|^{m}\big).\label{eq:D1}
\end{equs}
To estimate $D_2$, we use the definition of $R_\lambda$ from \eqref{eq:R lambda} and the fact that $\fra$, $\tfra$ are even  to write
\begin{equs}
D_2&\le \int_0^{|u|}\lambda^2+I_{|u|\geq R_\lambda}(1+\zeta^{m-1})\,d\zeta+\int_0^{|\tu|}\lambda^2+I_{|\tu|\geq R_\lambda}(1+\zeta^{m-1})\,d\zeta
\\
&\le \lambda^2\big(|u|+|\tu|)+I_{|u|\geq R_\lambda}(1+|u|^m)+I_{|\tu|\geq R_\lambda}(1+|\tu|^m).\label{eq:D2}
\end{equs}
Combining \eqref{eq:C1+C4}, \eqref{eq:D}, \eqref{eq:D1},  and \eqref{eq:D2},  we finally obtain
\begin{equs}
C_1 & +C_2 +C_3+C_4 
\le\eps^{-2}\big(\lambda^2 + \delta^{2\alpha} \big)
\E(1+\|u\|^{m}_{L_{m}(Q_T)}+\|\tu\|_{L_{m}(Q_T)}^{m})
\label{eq:C1through4}
\\
&\quad +
\eps^{-2}\E\|I_{|u|\geq R_\lambda}(1+|u|)\|_{L_m(Q_T)}^m+
\eps^{-2}\E\|I_{|\tu|\geq R_\lambda}(1+|\tu|)\|_{L_m(Q_T)}^m.
\end{equs}
Since one has
\begin{equ}
\big|\E\int_{t,x,y}\eta_\delta(u-\tilde{u})\partial_t\phi_\eps-\E\int_{t,x,y}|u-\tilde{u}|\partial_t\phi_\eps\big|\le\delta,
\end{equ}
from \eqref{eq:no-theta}, \eqref{eq:stochastic-terms}, and \eqref{eq:C1through4} one gets
\red{
\begin{equs}
-\E&\int_{t,x,y}|u(t,x)-\tilde{u}(t,y)|\varrho_\varepsilon(x-y)\D_t \varphi(t).
\\
\leq &
N\eps^{-2}\big(\E\|I_{|u|\geq R_\lambda}(1+|u|)\|_{L_m(Q_T)}^m+
\E\|I_{|\tu|\geq R_\lambda}(1+|\tu|)\|_{L_m(Q_T)}^m\big)
\\
& N\big(\delta^{2\kappa}+\eps^{2\bar\kappa}\delta^{-1}+\eps^{-2}\delta^{2\alpha}+\eps^{-2}\lambda^2+\delta^{-1}\sup_{r,x}\frac{|\sigma(x,r)-\tsigma(x,r)|^2_{l_2}}{(1+|r|)^{-(m+1)}})
\\
&\times \E(1+\|u\|^{m}_{L_{m}(Q_T)}+\|\tu\|_{L_{m}(Q_T)}^{m}). \label{eq:before eps limit}
\end{equs}
}
Denote the right-hand side above by $M$.
Let  $s, t \in (0,T)$, with $s < t$, be Lebesgue points of the function
$$
t \mapsto \E \int_{x,y}|u(t,x)-\tilde{u}(t,y)|\varrho_\varepsilon(x-y),
$$
and fix some $\gamma >0$ such that $\gamma< t-s$ and $t+ \gamma<T$. 
We now make use of the freedom of choosing $\varphi$:
choose in \eqref{eq:before eps limit} $\varphi =\varphi_n \in C^\infty_c((0,T))$ 
obeying the bound $\|\varphi\|_{L_\infty([0,T])}\vee\|\partial_t\varphi\|_{L_1([0,T])}\leq 1$,
such that 
$$
\lim_{n \to \infty} \|\varphi_n-\zeta\|_{H^1_0((0,T))} = 0,
$$
where 
$\zeta : [0,T] \to \bR$ is such that $\zeta(0)=0$ and $\zeta'=\gamma^{-1}I_{s,s+\gamma}-\gamma^{-1}I_{t,t+\gamma}$.
After letting $n \to \infty$ we obtain 
\begin{equs}
\frac{1}{\gamma}&\E\int_t^{t+ \gamma} \int_{x,y}|u(r,x)-\tilde{u}(r,y)|\varrho_\varepsilon(x-y) \, dr 
\\
&\leq M+ \frac{1}{\gamma}\E\int_s^{s+ \gamma} \int_{x,y}|u(r,x)-\tilde{u}(r,y)|\varrho_\varepsilon(x-y) \,  dr,
\end{equs}
which, after letting $\gamma \downarrow 0$,  gives
\begin{equs}
\E & \int_{x,y}|u(t,x)-\tilde{u}(t,y)|\varrho_\varepsilon(x-y) 
\\
&\leq M+ \E \int_{x,y}|u(s,x)-\tilde{u}(s,y)|\varrho_\varepsilon(x-y).
\end{equs}
Notice that the above inequality holds for almost all $s \leq t$. After averaging over $s \in (0,\gamma)$ for some $\gamma>0$ we obtain
\begin{equs}
\E & \int_{x,y}|u(t,x)-\tilde{u}(t,y)|\varrho_\varepsilon(x-y) 
\\
&\leq M+ \frac{1}{\gamma}\E\int_0^{\gamma} \int_{x,y}|u(s,x)-\tilde{u}(s,y)|\varrho_\varepsilon(x-y)\,ds.
\end{equs}
Letting $\gamma \to 0$, we obtain by virtue of Lemma \ref{lem:initial-condition},
\begin{equs}
\E & \int_{x,y}|u(t,x)-\tilde{u}(t,y)|\varrho_\varepsilon(x-y) 
\\
&\leq M+ \E \int_{x,y}|\xi(x)-\txi(y)|\varrho_\varepsilon(x-y).\label{eq:fork}
\end{equs}
To prove (ii), we integrate with respect to $t$ over $[0,T]$,
and write
\begin{equs}
\E & \int_{t,x,y}|u(t,x)-\tilde{u}(t,y)|\varrho_\varepsilon(x-y) 
\\
&\leq TM+T\sup_{|h|\leq\eps}\E\|\txi(\cdot)-\txi(\cdot+h)\|_{L_1(\T^d)}
+T\E \int_{x}|\xi(x)-\txi(x)|.
\end{equs}
Combining this with the fact that
\begin{equs}
\Big|\E & \int_{t,y}|u(t,y)-\tilde{u}(t,y)| -\E \int_{t,x,y}|u(t,x)-\tilde{u}(t,y)|\varrho_\varepsilon(x-y) \Big|
\\
&\leq
\E \int_{t,x,y}|u (t,x)-u(t,y)|\varrho_\varepsilon(x-y),
\end{equs}
and recalling Lemma \ref{lem:frac reg}, the estimate \eqref{eq:super inequality} is proved.

Moving on to (i), first note that in this case we can take $\lambda=0$ and  $R_\lambda=\infty$. 
Since also $\sigma=\tsigma$, we have
\begin{equ}
M=M(\eps,\delta)=N\big(\delta^{2\kappa}+\eps^{2\bar\kappa}\delta^{-1}+\eps^{-2}\delta^{2\alpha})
\E(1+\|u\|^{\red{m+1}}_{L_{\red{m+1}}(Q_T)}+\|\tu\|_{L_{\red{m+1}}(Q_T)}^{\red{m+1}}).
\end{equ}
We now choose $\nu$ such that
$\nu \in ((m\wedge2)^{-1},\bar\kappa)$ and then  $\alpha<1\wedge(m/2)$ such that
$-2+(2\alpha)(2\nu)>0$. Setting $\delta= \eps^{2\nu}$ then yields $M \to 0$ as $\eps \to 0$. Note that for a countable sequence $\eps_n\to 0$, inequality \eqref{eq:fork} holds for almost all $t$ for all $n$, and hence passing to the $n\to\infty$ limit, we get, using again the continuity of translations in $L_1$,
\begin{equs}
\E& \int_{x}|u(t,x)-\tilde{u}(t,x)|
\leq \E \int_{x}|\xi(x)-\tilde{\xi}(x)|
\end{equs}
for almost all $t$. Taking essential supremum in $t\in[0,T]$, we get \eqref{eq:L_1 contraction}.
\end{proof}

\begin{theorem}  \label{thm:uniqueness sqrt}
Assume the setting of Theorem \ref{thm:uniqueness}, replacing Assumption \ref{as:noise} \eqref{as:sigma} by Assumption \ref{as:noise} \eqref{as:sigma sqrt}.
Then,

(i) if furthermore $A=\tA$ and $\sigma=\tsigma$, then 
\begin{equ}\label{eq:L_1 contraction sqrt}
\esssup_{t\in[0,T]}\E\int_x|u(t,x)-\tu(t,x)|\leq\E\int_x|\xi(x)-\txi(x)|.
\end{equ}

(ii) for all $\eps,\delta\in(0,1]$, $\lambda\in[0,1]$ and $\alpha\in(0,1\wedge(m/2))$, we have
\begin{equs}
\E \int_{t,x} & |u(t,x)-\tilde{u}(t,x)|  
\leq T\E \int_{x}|\xi(x)-\txi(x)|
\\
&+N\eps^{\frac{2}{m+1}}\big(1+\E\|\nabla\Psi(u)\|_{L_1(Q_T)}\big)
\\
&+T\sup_{|h|\leq\eps}\big(\E\|\txi(\cdot)-\txi(\cdot+h)\|_{L_1(\T^d)}\big)
+N\delta^{-2}|\log\delta|^{-1}\sup_{r}\red{|\sigma(r)-\tsigma(r)|^2_{l_2}}
\\
&+N\eps^{-2}\E\big(\|I_{|u|\geq R_\lambda}(1+|u|)\|_{L_m(Q_T)}^m+
\|I_{|\tu|\geq R_\lambda}(1+|\tu|)\|_{L_m(Q_T)}^m\big)
\\
&+N\Big(|\log\delta|^{-1}+ \varepsilon^{-2} \delta^{2\alpha}+\eps^{-2}\lambda^2
\Big)
\\
&\quad\quad\times\E(1+\|u\|^{m}_{L_{m}(Q_T)}+\|\tu\|_{L_{m}(Q_T)}^{m})
\label{eq:super inequality 2}
\end{equs}
where $N$ depends only on $m,K,d,T,\alpha$, and $R_\lambda$ is as in \eqref{eq:R lambda}.
\end{theorem}

\begin{proof}
The main difference to the previous proof is the choice of $\eta_\delta$. We now take, similarly to \cite{YaWa},
\begin{equ}
\eta_\delta(0)=\eta_\delta^\prime(0)=0,\quad
\eta_\delta^{\prime\prime}(r)=\Big(\rho_{\delta^2/4}\ast\Big(\zeta\mapsto\frac{1}{\zeta|\log\delta|}I_{\zeta\in[\delta^2/2,\delta/2]}\Big)\Big)(|r|).
\end{equ}
We again have the first three properties in \eqref{eq:eta prop}. 
In place of the fourth, however, we now have $|\eta_\delta''(r)r|\leq 2|\log\delta|^{-1}$
and $|\eta_\delta''|\leq \delta^{-2}|\log\delta|^{-1}$.
Hence, in \eqref{eq:stochastic-terms} we get (recall that $\sigma$ does not depend on $x$)
\begin{equ}
C_5+C_6+C_7\lesssim|\log\delta|^{-1}+\delta^{-2}|\log\delta|^{-1}\sup_{r}\red{|\sigma(r)-\tsigma(r)|_{l_2}^2}.
\end{equ}
As pointed out in the previous proof, \eqref{eq:stochastic-terms} is the only point where a pointwise bound on $|\eta_\delta''|$ is used, and thus  the bound \eqref{eq:super inequality 2} can be obtained in the same way as \eqref{eq:super inequality}.

The proof of part (i) is also very similar, except for the passage to the $\eps,\delta\to 0$ limit.
As before, we take $\lambda=0$, $R_\lambda=\infty$, and
since the $\eps^{2\bar\kappa}\delta^{-1}$ term is now not present, there are no negative powers of $\delta$, so we may pass to the $\delta\to 0$ limit first.
This eliminates the $\eps^{-2}\delta^{2\alpha}$ term, after which no negative power of $\eps$ is left, so we then may pass to the $\eps\to 0$ limit. The proof is then concluded precisely as above.
\end{proof}

\section{Approximation and proof of Theorems \ref{thm:main}-\ref{thm:stability} }
To approximate the coefficients of \eqref{eq:main_equation}, we take 
\begin{equ}        \label{eq:apr-xi}
\red{ \sigma_n:=\rho_{1/n}^{\otimes(d+1)}\ast \bar \sigma_n,\, \bar \sigma_n(x,r)=\sigma(x,-n \vee( r\wedge n)), } \, 
\xi_n:=\rho_{1/n}^{\otimes d}\ast (-n\vee(\xi\wedge n)).
\end{equ}
If $\sigma$ and $\xi$ satisfy Assumption \ref{as:noise} (\ref{as:sigma}) (or (\ref{as:sigma sqrt})) and  Assumption \ref{as:A} (\ref{as:ic}), respectively, with a constant $K\geq1$, then the same holds for $\sigma_n$ and $\xi_n$, with, say, constant $2K$.
It is also clear that $\sigma_n\in C^\infty(\T^d\times\R)$ with bounded derivatives. 
\red{The following red part should be removed: Indeed, for any $x \in \bT^d$, $k\in\N$, and $r\in[k,k+1]$, one can write
\begin{equs}
|\D_r\sigma_n(x,r)|_{l_2}&=\big|\D_r\big((\rho_{1/n}^{\otimes(d+1)}\ast\sigma(x,r)-\sigma(x,k)\big)\big|_{l_2}
\\
&\leq n^{d+2}\sup_{r\in[k,k+2]}|\sigma(x,r)-\sigma(x,k)|_{l_2}\leq 2K n^{d+2}.
\end{equs}
Similar argument shows the uniform bound on $|\nabla_x\sigma_n(x,r)|$.}
Also, $\xi_n$ is a bounded  $\mathcal{F}_0$-measurable $C^{k}(\bT^d)$-valued random variable for any $k\in\N$, and
\begin{equ}\label{eq:approx sigma xi}
\red{\sup_{x,r}\frac{|\sigma(x,r)-\sigma_n(x,r)|^2_{l_2}}{(1+|r|)^{m+1}} }\underset{n\rightarrow\infty}{\rightarrow}0,
\quad
\E\|\xi-\xi_n\|_{L_{m+1}(\T^d)}^{m+1}\underset{n\rightarrow\infty}{\rightarrow}0.
\end{equ}
The approximation of $A$ is a bit more subtle so let us state it separately.
\begin{proposition}
Let $A$ satisfy Assumption \ref{as:A} (\ref{as:A first}) with a constant $K\geq 1$.
Then, for all $n$  there exists a function $A_n\in C^\infty(\R)$ with bounded derivatives,
satisfying Assumption \ref{as:A} (\ref{as:A first}) with constant $3K$,
such that $\fra_n(r)\geq 2/n$, and 
\begin{equ}\label{eq:approx A}
\sup_{|r|\leq n}|\fra(r)-\fra_n(r)|\leq 4/n.
\end{equ}
\end{proposition}
\begin{proof}
Take a symmetric mollifier $\bar\rho_\eps$ supported on $[-\eps, \eps]$, for instance $\bar\rho_\eps(r):=\int_\R\rho_\eps(r+s)\rho_\eps(s)\,ds$.
Define $\eps_n:=\sup\{\eps\in(0,1]:\,|\fra(r)-\fra(\zeta)|\leq 1/n,\,\forall|r|\leq 3n,\,|\zeta-r|\leq3\eps\}>0$.
We then claim that 
\begin{equ}
A_n(r)=\int_0^r\fra_n^2(\zeta)\,d\zeta,\quad
\fra_n(r)=\bar\rho_{\eps_n}\ast\big(2/n+\fra(3\eps_n\vee |r| \wedge 3n)\big),
\end{equ}
behaves as stated.

It is trivial that $A_n$ is smooth, has bounded derivatives, and that $\fra_n\geq 2/n$.
The bound \eqref{eq:approx A} follows from the definition of $\eps_n$.
To verify Assumption \ref{as:A} (\ref{as:A first}), the first bound in \eqref{eq:as fra} is obvious.
As for the second, we have $\fra_n'(r)=0$ for $|r|\leq 2\eps_n$, while for $|r|>2\eps_n$, one has
\begin{equ}
|\fra_n'(r)|\leq\sup_{\zeta\in[r-\eps_n,r+\eps_n]}|\fra'(\zeta)|\leq\sup_{\zeta\in[r/2,2r]}|\fra'(\zeta)|\leq 2K|r|^{\frac{m-3}{2}}.
\end{equ}
For \eqref{eq:as Psi} notice that by choosing $\eps_n$ as above, we have that on $[-2n,2n]$, $\fra_n\geq \fra$.
This easily implies the bounds $K\fra_n(r)\geq I_{|r|\geq 1}$ and
\begin{equ}
K|\Psi_n(r)-\Psi_n(\zeta)|
  \geq K|\Psi(r)-\Psi(\zeta)|, \quad \text{if}\ |r|\vee|\zeta|\leq 2.
\end{equ} 
If $|r|\vee|\zeta|\geq 2$, we separate three cases:
\begin{enumerate}[(i)]
\item if $|r|\wedge|\zeta|\geq 1$, and $r$ and $\zeta$ have the same sign, then simply from 
$K\fra_n(r)\geq I_{|r|\geq 1}$ we get $K|\Psi_n(r)-\Psi_n(\zeta)|\geq |r-\zeta|$.
\item if $|r|\wedge|\zeta|\geq 1$, and $r$ and $\zeta$ have the opposite sign, then by symmetry we may assume $r\geq 2$, $\zeta\leq-1$. We can then write
\begin{equs}
K|\Psi_n(r)-\Psi_n(\zeta)|&\geq K\big(\Psi_n(r)-\Psi_n(1)\big)+K\big(\Psi_n(-1)-\Psi_n(\zeta)\big)
\\
&\geq r-1-1-\zeta=r-\zeta-2\geq (1/3)(r-\zeta).
\end{equs}
\item if $|r|\wedge|\zeta|\leq 1$, then by symmetry we may assume $r\geq2$, $|\zeta|\leq 1$. We can then write
\begin{equ}
K|\Psi_n(r)-\Psi_n(\zeta)|\geq K\big(\Psi_n(r)-\Psi_n(1)\big)\geq r-1\geq(1/3)(r-\zeta).
\end{equ}
\end{enumerate}
The proof is finished.
\end{proof}

Taking $A_n$, $\sigma_n$, and $\xi_n$ as above, by \cite[Sec~4]{DHV16}, for every $n$ the Cauchy problem $\cE(A_n,\sigma_n,\xi_n)$ has a unique solution $u_n$ in the $L_2$ sense, that is,  $u_n$ is a predictable,   continuous $L_2(\T^d)$-valued process,
$u_n\in L_2(\Omega_T, W^1_2(\T^d))$, $\nabla A_n(u_n)\in L_2(\Omega_T, L_2(\T^d))$, and the equality
\begin{equ}
(u_n(t,\cdot),\phi)=(\xi_n,\phi)
-\int_0^t(\nabla A_n(u_n(s,\cdot)),\nabla \phi)\,ds
+\int_0^t(\sigma_n^k(\cdot,u_n(s,\cdot)),\phi)\,d\beta^k_s\,
\end{equ}
holds for all $\phi\in C^\infty(\T^d)$, almost surely for all $t\in[0,T]$.
For such processes
It\^o's formula holds for $\|\cdot\|_{L_2(\T^d)}^2$ and $\|\cdot\|_{L_{m+1}(\T^d)}^{m+1}$, (see \cite[Sec~3]{K_Ito} and the approximation argument in \cite[Lemma 2]{DG15}, respectively) which implies the uniform estimates
\begin{equs}
\E \sup_{t \leq T} \| u_n\|_{L_2(\bT^d)}^p + \E \|\nabla \Psi_n  (u_n) \|_{L_2(Q_T)}^p &\leq N ( 1+ \E \|\xi_n\|_{L_2(\bT^d)}^p)
\\
\E \sup_{t \leq T} \|u_n\|_{L_{m+1}(Q_T)}^{m+1} &\leq N(1+\E \|\xi_n\|_{L_{m+1}(\bT^d)}^{m+1}),
\end{equs}
for all $p \geq 2$.
Applying also It\^o's formula for the function $u \to \int_{x} \int_0^u A_n(s) \, ds $ and using the above estimate  yields
\begin{equs}         
\E \| \nabla A_n(u_n)\|_{L_2(Q_T)}^2  \leq N(1+\E \|\xi_n\|_{L_{m+1}(\bT^d)}^{m+1}).
\end{equs}
In these estimates, the constant $N$ depends only on $K, T, d,p$ and $m$ (but not on $n\in \bN$). Notice that $\xi_n$ are bounded by $n$, which implies that the right hand side of the above inequalities is finite. Moreover, by construction of $\xi_n$ one concludes that for all $p\geq 2$
\begin{equs}\label{eq:uniform}
\E \sup_{t \leq T} \| u_n\|_{L_2(\bT^d)}^p + \E \|\nabla \Psi_n  (u_n) \|_{L_2(Q_T)}^p &\leq N ( 1+ \E \|\xi\|_{L_2(\bT^d)}^p),
\\
\label{eq:uniform-m+1}
\E \sup_{t \leq T} \|u_n\|_{L_{m+1}(Q_T)}^{m+1}+ \E \| \nabla A_n(u_n)\|_{L_2(Q_T)}^2 &\leq N(1+\E \|\xi\|_{L_{m+1}(\bT^d)}^{m+1}).
\end{equs}
with $N$ depending only on $K, T, d,p$ and $m$.
Finally, since $\fra_n\geq 2/n>0$, we have $|\nabla u_n|\leq N(n)|\nabla\Psi_n(u_n)|$, and so by \eqref{eq:uniform}, we have the ($n$-dependent) bound
\begin{equ}\label{eq:gradient}
\E\|\nabla u_n\|_{L_2(Q_T)}^p<\infty.
\end{equ}

\begin{lemma}\label{lem:strong entropy}
The functions $u_n$ above, have the $(\star)$-property with coefficient $\sigma_n$. If moreover $\|\xi\|_{L_{2}(\bT^d)}$ has finite moments up to order $4$, then the constant $N$ in \eqref{eq:strong-entropy} is independent of $n$.
\end{lemma}

\begin{proof}
Fix $\theta>0$ small enough so that \eqref{eq:F rewrite} holds. To ease notation we drop the lower index in $F_\theta$.
We proceed by two approximations:  first, as in Corollary \ref{cor} (ii),  the substitution of $u_n(t,x)$ into
$F(t,x,\cdot)$ is smoothed, and second,  $u_n$ is regularised. 

For a function $f \in L_2(\bT^d)$ let  $f^{(\gamma)}:=(\rho_\gamma)^{\otimes d}\ast f$ denote   its mollification. 
Then,  $u_n^{(\gamma)}$ satisfies (pointwise) the equation
\begin{equ}\label{eq:viscous smoothed equation}
du_n^{(\gamma)}=\Delta (A_n(u_n))^{(\gamma)}\,dt+(\sigma^k_n(u_n))^{(\gamma)}\,d\beta^k(t).
\end{equ}
We note that 
\begin{equs}
& \Big|\E\int_{t,x,a} F(t,x, a) \rho_\lambda(u_n(t,x)-a)- \E\int_{t,x,a} F(t,x, a) \rho_\lambda(u_n^{(\gamma)}(t,x)-a)\Big|
\\
&=\Big|\E\int_{t,x,a}\big(F(t,x,a)-F(t,x,a+u_n^{(\gamma)}(t,x)-u_n(t,x))\big)\rho_\lambda(u_n(t,x)-a)\Big|
\\
&\leq N  \left( \E \|u_n-u_n^{(\gamma)} \|_{L_1(Q_T)}^2\right)^{1/2}  \,  \left( \E\|\D_a F\|_{L_\infty(Q_T\times\R)}^2
\right)^{1/2}  \to 0, \label{eq:0gamma limit}
\end{equs}
as $\gamma \to 0$.
By \eqref{eq:F rewrite} we have  $\E F(t,x,a)X=0$ for any $\cF_{t-\theta}$-measurable bounded random variable $X$. Hence, 
\begin{equs}
\E &F(t,x, a) \rho_\lambda(u_n^{(\gamma)}(t,x)-a)
\\
&=\E F(t,x, a) [\rho_\lambda(u_n^{(\gamma)}(t,x)-a)-\rho_\lambda(u_n^{(\gamma)}(t-\theta,x)-a)].
\end{equs}
 By \eqref{eq:viscous smoothed equation} and It\^o's formula one has
\begin{equs}
&\int_{t,x,a} F(t,x, a) \big(\rho_\lambda (u_n^{(\gamma)} (t,x)-a)-\rho_\lambda(u_n^{(\gamma)}(t-\theta,x)-a)\big)
\\
&=
\int_{t,x,a} F(t,x, a) \int_{t-\theta}^t \rho^{\prime}_\lambda  (u_n^{(\gamma)} (s,x)-a)
 \Delta (A_n(u_n))^{(\gamma)} \, ds 
\\
&\quad+
\int_{t,x,a} F(t,x, a) \int_{t-\theta}^t \rho^{\prime}_\lambda  (u_n^{(\gamma)} (s,x)-a) (\sigma^k(x,u_n(s,x)))^{(\gamma)} \, d\beta^k(s) 
\\
&\quad+
\int_{t,x,a} F(t,x, a) \frac{1}{2}\int_{t-\theta}^t \rho^{\prime\prime}_\lambda  (u_n^{(\gamma)} (s,x)-a) \sum_{k=1}^\infty|(\sigma^k(x,u_n(s,x)))^{(\gamma)}|^2 \, ds 
\\&=: C^{(1)}_{\lambda,\gamma}+C^{(2)}_{\lambda,\gamma}+C^{(3)}_{\lambda,\gamma}.        \label{eq:after Ito}
\end{equs}
By \eqref{eq:F rewrite} and integration by parts (in $x$) we have 
\begin{equs}
-C^{(1)}_{\lambda,\gamma}&= \int_{t,x,a}I_{t> \theta}\int_{t-\theta}^t  \nabla_x F(t,x, a)\rho^{\prime}_\lambda  (u_n^{(\gamma)} (s,x)-a)\nabla(A_n(u_n))^{(\gamma)}
\\
&\quad+ F(t,x, a)\rho^{\prime\prime}_\lambda  (u_n^{(\gamma)} (s,x)-a) \nabla u_n^{(\gamma)}(s,x) 
\nabla(A_n(u_n))^{(\gamma)}\, ds 
\\
& =: C^{(11)}_{\lambda,\gamma}+C^{(12)}_{\lambda,\gamma} .
\end{equs}
\bigskip
Note that
\begin{equ}            \label{eq:whole-theta}
\int_\theta^T\int_{t-\theta}^t|f(s)|\,ds\,dt\leq\theta\int_{0}^T|f(s)|\,ds.
\end{equ}
Hence, after integration by parts with respect to $a$, by the Cauchy-Schwarz inequality and Lemma \ref{lem:F}, we have
\begin{equs}
\E|C^{(11)}_{\lambda,\gamma}|
&=\E\big|\int_{t,x,a}I_{t> \theta}\int_{t-\theta}^t\nabla_x\partial_a F(t,x,a)\rho_\lambda(u_n^{(\gamma)}(s,x)-a)\nabla(A_n(u_n))^{(\gamma)}\,ds\big|
\\
&\leq N\theta \left( \E\|\nabla_x\partial_a F\|^2_{L_\infty(Q_T\times\R)} \right) ^{1/2}
\left( \E\| \nabla A_n(u_n)\|^2_{L_1(Q_T)} \right) ^{1/2}
\\
&\leq N(n)\theta^{1-\mu}.                    \label{eq:A1}
\end{equs}
Similarly, this time integrating by parts twice in $a$ we have 
\begin{equs}
\E&|C^{(12)}_{\lambda,\gamma}|
\leq
N\theta^{1-\mu}
\left( \E
\| \nabla u_n^{(\gamma)}\nabla (A_n(u_n))^{(\gamma)}\|_{L_1(Q_T)}^{\frac{m+1}{m}} \right)^{\frac{m}{m+1}}.
\end{equs}
To bound the right-hand side, note that by \eqref{eq:gradient},
$\nabla u_n^{(\gamma)}\to\nabla u_n$ in $L_p(\Omega;L_2(Q_T))$, for any $p$,
and by \eqref{eq:uniform-m+1},
$\nabla (A_n(u_n))^{(\gamma)}\to\nabla A_n(u_n)$
 in $L_2(\Omega;L_2(Q_T))$. Therefore,
\begin{equs}
\lim_{\gamma\to 0}\E
\| \nabla u_n^{(\gamma)}\nabla (A_n(u_n))^{(\gamma)}\|_{L_1(Q_T)}^{\frac{m+1}{m}}
&=\E\| \nabla u_n\nabla A_n(u_n)\|_{L_1(Q_T)}^{\frac{m+1}{m}}
\\& 
=\E\|\nabla \Psi_n(u_n)\|_{L_2(Q_T)}^{\frac{2(m+1)}{m}}\leq N(n). \label{eq:A2}
\end{equs}
Together with \eqref{eq:A1}, we therefore get
\begin{equ}\label{eq:A}
\limsup_{\gamma\to 0}\E|C^{(1)}_{\lambda,\gamma}|\leq N(n) \theta^{1-\mu}.
\end{equ}
To bound $C^{(3)}_{\lambda,\gamma}$, we proceed similarly:
 integrate by parts twice with respect to $a$, use Lemma \ref{lem:F}, and \eqref{eq:whole-theta} to get 
\begin{equation}\label{eq:C}
\E|C^{(3)}_{\lambda,\gamma}|
\leq N\theta\left( \E\|\partial_{aa} F\|^2_{L_\infty(Q_T\times\R)}\right) ^{1/2}\left( \E\big\||\sigma_n(u_n)\big|_{l_2}\|^4_{L_2(Q_T)}\right)^{1/2} \leq N(n)\theta^{1-\mu}.
\end{equation}
Next, one has by It\^o's isometry 
\begin{equs}
\E C^{(2)}_{\lambda,\gamma}&=\E\int_{a,t,x,y}\int_{t-\theta}^t
\tsigma^k(y,\tu(s,y))\big(\sigma^k_n(x,u_n(s,x))\big)^{(\gamma)}g(\tu(s,y)-a)
\\
&\quad\quad\times\rho_\lambda'(u_n^{(\gamma)}(s,x)-a)\phi_\theta(t,x,s,y)\,ds
\\
&=-\E\int_{a,t,x,y}\int_{t-\theta}^t
\tsigma^k(y,\tu(s,y))\big(\sigma^k_n(x,u_n(s,x))\big)^{(\gamma)}g'(\tu(s,y)-a)
\\
&\quad\quad\times\rho_\lambda(u_n^{(\gamma)}(s,x)-a)\phi_\theta(t,x,s,y)\,ds.
\end{equs}
Now the passage to the $\gamma\to 0$ limit is straightforward.
Passing to the $\lambda\rightarrow0$ limit is also quite immediate, since, 
\begin{equs}\label{eq:B limit lambda}
\E&|C^{(2)}_{\lambda,0}-C^{(2)}_{0,0}|
\\&\leq\lambda\sup_a|g''(a)|\E\int_{t,x,y}\int_{t-\theta}^t|\tsigma(y,\tu(s,y))|_{l_2}|\sigma_n(x,u_n(s,x))\phi_\theta(t,x,s,y)|_{l_2}\,ds.
\end{equs}
Putting all of \eqref{eq:0lambda limit}, \eqref{eq:0gamma limit}, \eqref{eq:after Ito}, \eqref{eq:A}, \eqref{eq:C}, and \eqref{eq:B limit lambda} together, we conclude
\begin{equs}
\E &\int_{t,x} F(t,x, v(t,x))
\\
&\leq
\limsup_{\lambda\to0}\limsup_{\gamma\to 0}\E\big(
|C^{(1)}_{\lambda,\gamma}|+|C^{(3)}_{\lambda,\gamma}|\big) 
+ \lim_{\lambda\to0}\lim_{\gamma\to 0} \E C^{(2)}_{\lambda,\gamma}
\\
&\leq N(n) \theta^{1-\mu}
 -  \E \int_{s,t,x,y} \tsigma^k(y, \tilde{u}(s,y))\sigma^k_n(x, u_n(s,x)) g'(\tilde{u}(s,y)-u_n(s,x))\phi_{\theta},
\end{equs}
as claimed. Moreover, if $\E\|\xi\|^4_{L_2(\bT^d)}< \infty$, then by virtue of \eqref{eq:uniform} and \eqref{eq:uniform-m+1} it is clear that in  \eqref{eq:A1}, \eqref{eq:A2}, and \eqref{eq:C}, we can choose $N$ independent of $n \in \bN$, which completes the proof.
\end{proof}

We are ready to proceed with the proof of our main theorem.
\begin{proof}[Proof of Theorem \ref{thm:main}]
We only give the details for the case that $\sigma$ satisfies Assumption \ref{as:noise} \eqref{as:sigma}. The statement in the case of Assumption \ref{as:noise} \eqref{as:sigma sqrt} is obtained analogously, making use of Theorem \ref{thm:uniqueness sqrt} instead of Theorem \ref{thm:uniqueness}.

\emph{Existence:} First assume that 
$\E \|\xi\|_{L_2(\bT^d)}^4< \infty$.
 We will show that $(u_n)_{n\in\N}$ is a Cauchy sequence in $L_1(\Omega_T;L_1(\T^d))$.
Set $\eps_0>0$ arbitrary.
As in the conclusion of the proof of Theorem \ref{thm:uniqueness}, we choose $\nu$ such that
$\nu \in ((m\wedge2)^{-1},\bar\kappa)$ and then we choose $\alpha<1\wedge(m/2)$ such that
$-2+(2\alpha)(2\nu)>0$.
Then, we apply Theorem \ref{thm:uniqueness} (ii) to $u_n$ and $u_{n'}$, for arbitrary $n\leq n'$, setting
$\delta=\eps^{2\nu}$, and $\lambda=8/n$. Thanks to \eqref{eq:approx A}, we have that $R_\lambda\geq n$.
Recalling the uniform estimates \eqref{eq:uniform}, and the triangle inequality
\begin{equ}
\E\|\xi_{n'}(\cdot)-\xi_{n'}(\cdot+h)\|_{L_1(\T^d)}\leq
\E\|\xi(\cdot)-\xi(\cdot+h)\|_{L_1(\T^d)}+2\E\|\xi-\xi_{n'}\|_{L_1(\T^d)},
\end{equ}
the right-hand side of \eqref{eq:super inequality} is bounded by
\begin{equs}
M(\eps)&+N\E\|\xi-\xi_{n'}\|_{L_1(\T^d)}+N \E\|\xi-\xi_{n}\|_{L_1(\T^d)}
\\
&+N\eps^{-2 \nu} \red{\sup_{x,r}\frac{|\sigma_n(x,r)-\sigma_{n'}(x,r)|_{l_2}^2}{(1+|r|)^{m+1}} }+N\eps^{-2}n^{-2}
\\
&+N\eps^{-2}\E\big(\|I_{|u_n|\geq n}(1+|u_n|)\|_{L_m(Q_T)}^m+
\|I_{|u_{n'}|\geq n}(1+|u_{n'}|)\|_{L_m(Q_T)}^m\big),
\end{equs}
where $M(\eps)\to 0$ as $\eps\to 0$. Let $\eps>0$ be such that $M(\eps)\leq\eps_0$.
By \eqref{eq:approx sigma xi}, one can then choose $n_0$ sufficiently large so that for $n_0\leq n\leq n'$ the second through fifth terms above are each bounded by $\eps_0$.
The same is true for the last term, thanks to the uniform integrability (in $(\omega,t,x)$) of $1+|u_n|^m$, which in turn follows from \eqref{eq:uniform-m+1}. Hence, indeed, for $n_0\leq n\leq n'$, one has
\begin{equ}
\E \int_{t,x} |u_n(t,x)-u_{n'}(t,x)|\leq 6\eps_0.
\end{equ}
Therefore, $(u_n)_{n\in\N}$ converges in $L_1(\Omega_T; L_1(\T^d))$ to a limit $u$.
 Moreover, by passing to a subsequence, we may also assume that 
 \begin{equation}               \label{eq:ae-convergence}
 \lim_{n \to \infty}u_n=u, \qquad \text{for almost all} \qquad (\omega, t, x)\in \Omega_T \times \bT^d. 
 \end{equation}
Consequently, by Lemma \ref{lem:strong entropy}, \eqref{eq:uniform-m+1},  and Corollary \ref{cor} (i), we have that $u$ has the $(\star)$-property with coefficient $\sigma$. 
In addition, it follows by \eqref{eq:uniform-m+1} that  for  any $q < m+1$, 
\begin{equation}            \label{eq:uniform-integrability}
(|u_n(t,x)|^q)_{n=1}^\infty \ \text{is uniformly integrable on $\Omega_T \times \bT^d$}.
\end{equation} 

We now show that $u$ is an entropy solution.  From now on, when we refer to the estimates \eqref{eq:uniform}, we only use them with $p=2$. 
 By the estimates in \eqref{eq:uniform-m+1}, it follows that $u$ satisfies \eqref{item:in-Lm} from Definition \ref{def:solution}.
 
 For $f \in C_b(\bR)$ and $\eta$ as in Definition \ref{def:solution} we define $\Psi_{n,f}$ and $q_{n,\eta}$ analogously to $\Psi_f$ and $q_\eta$, but with $\mathfrak{a}_n$ in place of $\mathfrak{a}$.    For each $n$, we clearly have $\Psi_{n,f}(u_n) \in L_2(
\Omega_T; W^1_2(\bT^d))$ and $\D_i \Psi_{n,f}(u_n)= f(u_n) \D_i \Psi_n(u_n)$. Also, we have $|\Psi_{f,n}(r)|\leq \|f\|_{L_\infty} 3K |r|^{(m+1)/2}$  for all $r \in \bR$, which combined with  \eqref{eq:uniform} and \eqref{eq:uniform-m+1} gives that that
\begin{equation*}
\sup_n \E\int_t \|\Psi_{n,f}(u_n)\|_{W^1_2(\bT^d)}^2 \, < \infty. 
\end{equation*} 
Hence, for a subsequence we have  $\Psi_{n,f}(u_n) \wto v_f$, $\Psi_n(u_n) \wto v$ for some $v_f, v \in L_2(\Omega_T; W^1_2(\bT^d))$. By \eqref{eq:approx A} and  \eqref{eq:ae-convergence}-\eqref{eq:uniform-integrability} it is easy to see that 
$v_f= \Psi_f(u)$, $ v= \Psi(u)$. Moreover, for any $\phi \in C^\infty(\bT^d)$, $B \in \mathcal{F}$, we have 
\begin{equs}
\E I_B \int_{t,x}  \D_i \Psi_f(u) \phi \,  &= \lim_{n \to \infty} \E I_B \int_{t,x} \D_i \Psi_{f,n}(u_n) \phi \, 
\\
&= \lim_{n \to \infty} \E I_B \int_{t,x} f(u_n)\D_i\Psi_n(u_n) \phi \, 
\\
&= \E I_B \int_{t,x} f(u)\D_i\Psi (u) \phi \, ,
\end{equs}
where for the last equality we have used that $ \D_i\Psi_n(u_n) \wto  \D_i\Psi(u)$ (weakly) and $f(u_n) \to f(u)$ (strongly) in $L_2(\Omega_T; L_2(\bT^d))$. Hence, \eqref{item:chain_ruleW2} from Definition \ref{def:solution} is also satisfied. We now show \eqref{item:entropies}. Let $\eta$ and $\phi$ be as in \eqref{item:entropies} and let $B \in \mathcal{F}$. By It\^o's formula (see, e.g., \cite{K_Ito}) for the function 
$$
u \mapsto \int_x \eta(u)\varrho,
$$
 and It\^o's product rule, we have 
\begin{equs}            
\nonumber          
-\E I_B \int_{t,x} \eta(u_n)\D_t\phi \, 
&\leq \E I_B \left[ \int_x\eta(\xi_n)\phi(0)+\int_{t,x} q_{n,\eta}(u_n) \Delta \phi  \,  \right.
\\
\nonumber   
&-\int_{t,x}\phi \eta''(u_n) | \nabla \Psi_n(u_n)|^2 \,  +\frac{1}{2}\int_{t,x} \phi \eta''(u_n)|\sigma_n(u_n)|_{l_2}^2 \, 
\\   \label{eq:un-entropy-inequality}
&+ \left. \int_0^T \int_x \phi \eta'(u_n)\sigma_n^k(u_n) \, d\beta^k(t)  \right].
\end{equs}
On the basis of \eqref{eq:ae-convergence}-\eqref{eq:uniform-integrability} and the construction of $\xi_n$, $\sigma_n$ and $\mathfrak{a}_n$ it is easy to see that 
\begin{equs}
&\lim_{n \to \infty} \E I_B \int_{t,x} \eta(\xi_n) \D_t\phi \,  =\E I_B \int_{t,x} \eta(\xi) \D_t\phi \, 
\\
&\lim_{n \to \infty} \E I_B \int_{t,x} \eta(u_n) \D_t\phi \,  =\E I_B \int_{t,x} \eta(u) \D_t\phi \, 
\\
&\lim_{n \to \infty}\E I_B \int_{t,x} q_{n,\eta}(u_n) \Delta \phi  \,  = \E I_B \int_{t,x} q_\eta(u) \Delta \phi \, 
\\
& \lim_{n \to \infty}\E I_B \int_{t,x} \phi \eta^{\prime \prime} (u_n)|\sigma_n(u_n)|^2 \,  =\E I_B \int_{t,x} \phi \eta^{\prime \prime} (u)|\sigma(u)|^2 
\\
& \lim_{n \to \infty}\E I_B \int_0^T \int_x \phi \eta^{\prime} (u_n)\sigma^k_n(u_n)\, d\beta^k(t)=\E I_B \int_0^T \int_x \phi \eta^{\prime}(u)\sigma^k(u) \, d\beta^k(t).
\end{equs}
Let us set $\tilde{f}(r):= \sqrt{\eta''(r)}$. Notice that $\D_i \Psi_{\tilde{f},n}(u_n)= \sqrt{\eta''(u_n)} \D_i \Psi_n(u_n)$. As before we have (after passing to a subsequence if necessary)  $\D_i \Psi_{\tilde{f},n}(u_n) \wto \D_i \Psi_{\tilde{f}}(u)$ in $L_2(\Omega_T ;L_2(\bT^d))$. In particular, this implies that $\D_i \Psi_{\tilde{f},n}(u_n) \wto \D_i \Psi_{\tilde{f}}(u)$ in $L_2(\Omega_T\times \bT^d , \bar\mu)$, where $d \bar\mu:= I_B\phi  \  d\bP\otimes dx \otimes dt $. This implies that 
\begin{equs}
\E I_B\int_{t,x}\phi \eta''(u) | \nabla \Psi(u)|^2\,  \leq \liminf_{n \to \infty} \E I_B \int_{t,x} \phi \eta''(u_n) | \nabla \Psi_n(u_n)|^2\, .
\end{equs}
Hence, taking $\liminf$ in \eqref{eq:un-entropy-inequality} along an appropriate subsequence, we see that $u$ satisfies also \eqref{item:entropies}.

 To summarise, we have shown that if in addition to the assumptions of Theorem \ref{thm:main} we have that $\E\| \xi\|_{L_2(\bT^d)}^4< \
\infty$, then there exists an entropy solution to \eqref{eq:main_equation} which has the $(\star)$-property with coefficient $\sigma$ (therefore, it is also unique by Theorem \ref{thm:uniqueness}). In addition, we can pass to the limit in \eqref{eq:uniform}-\eqref{eq:uniform-m+1} to obtain that 
\begin{equation}   \label{eq:uniform2}
\begin{aligned}
\E \sup_{t \leq T} \| u\|_{L_2(\bT^d)}^2 + \E \|\nabla \Psi  (u) \|_{L_2(Q_T)}^2 &\leq N ( 1+ \E \|\xi\|_{L_2(\bT^d)}^p),
\\
\E \sup_{t \leq T} \|u\|_{L_{m+1}(Q_T)}^{m+1}+ \E \| \nabla A(u)\|_{L_2(Q_T)}^2 &\leq N(1+\E \|\xi\|_{L_{m+1}(\bT^d)}^{m+1}),
\end{aligned}
\end{equation}
with a constant $N$ depending only on $d,K,T$ and $m$.

We now remove the extra condition on $\xi$. For $n \in \bN$, let $\xi_n$ be as in \eqref{eq:apr-xi} and let $u_{(n)}$ be the unique solution of $\mathcal{E}(A, \sigma, \xi_{n})$. 
Notice that $u_{(n)}$ has the $(\star)$-property with coefficient $\sigma$. Hence, by Theorem \ref{thm:uniqueness} (i) we have that $(u_{(n)})$ is Cauchy in $L_1(\Omega_T;L_1(\bT^d))$ and therefore has a limit $u$. In addition, $u_{(n)}$ satisfy the estimates \eqref{eq:uniform2} uniformly in $n \in \bN$. With the arguments provided above it is now  routine to show that 
$u$ is an entropy solution. 

\emph{Uniqueness:} We finally show \eqref{eq:main contraction} which also implies uniqueness. Let $\tilde{u}$ be an entropy solution of $\mathcal{E}(A,\sigma, \tilde{\xi})$. By Theorem \ref{thm:uniqueness} we have 
$$
\esssup_{t\in[0,T]}\E\int_x|u_{(n)}(t,x)-\tu(t,x)|\leq\E\int_x|\xi^n(x)-\txi(x)|,
$$
where $u_{(n)}$ are as above. We then let $n \to \infty$ 
to finish the proof.
\end{proof}
\begin{proof}[Proof of Theorem \ref{thm:stability}]
For each $n\in\N$, let $\hat \xi_n=-N_0\vee(\xi_n\wedge N_0)$, where $N_0=N_0(n)$ is chosen so that
$\E\|\hat\xi_n-\xi_n\|_{L_1(\T^d)}\leq 1/n$.
Furthermore, let $\hat u_n$ denote the entropy solution of $\cE(A_n,\sigma_n,\hat\xi_n)$, which, by the preceding, exists, is unique, and has the $(\star)$-property.
Since by \eqref{eq:main contraction} we have that
\begin{equ}
\E\|\hat u_n-u_n\|_{L_1(Q_T)}\leq T/n,
\end{equ}
it suffices to check that $\hat u_n\to u$ in $L_1(\Omega_T\times\T^d)$. This however follows from the bounds \eqref{eq:super inequality}, \eqref{eq:super inequality 2}, precisely as in the previous proof.
\end{proof}

\section*{Acknowledgment}
B. Gess acknowledges financial support by the DFG through the CRC 1283 ``Taming uncertainty and profiting from randomness and low regularity in analysis, stochastics and their applications".



\def\cprime{$'$}

\end{document}